\documentclass[11pt,a4paper,review]{article}
\newcommand{\weak}{\rightharpoonup}
\newcommand{\weakstar}{\stackrel{\ast}{\rightharpoonup}}

\newcommand{\vect}[1]{\boldsymbol{#1}}

\newcommand{\dx}{\mathrm{d}x}
\newcommand{\dy}{\mathrm{d}y}
\newcommand{\ds}{\mathrm{d}s}
\newcommand{\divv}[1]{\mathrm{div}(#1)}

\newcommand{\tens}[1]{\pmb{\mathsf{#1}}}
\newcommand{\tE}{\tens{E}}

\newcommand{\vu}{\vect{u}}
\newcommand{\vv}{\vect{v}}
\newcommand{\vf}{\vect{f}}
\newcommand{\vX}{\vect{V}_h}
\newcommand{\vXD}{\vect{V}_{D,h}}
\newcommand{\mR}{\mathcal{R}}
\newcommand{\mH}{\mathcal{H}_\gamma}
\newcommand{\mHh}{\mathcal{H}_{\gamma,h}}
\newcommand{\tS}{\tens{S}}

\newcommand{\Hu}{\vect{V}_D}
\newcommand{\fr}{\mathcal{H}^{p}_{\gamma}}
\newcommand{\frF}{\mathcal{H}^{F}_{\gamma}}
\newcommand{\frh}{\mathcal{H}^p_{\gamma,h}}
\newcommand{\frFh}{\mathcal{H}^{F}_{\gamma,h}}

\newcommand{\ffrF}{\tilde{\mathcal{H}}^{F}_{\gamma}}
\newcommand{\ffrFh}{\tilde{\mathcal{H}}^{F}_{\gamma, h}}
\newcommand{\mS}{\mathcal{S}}
\newcommand{\frho}{\tilde{\rho}}
\newcommand{\feta}{\tilde{\eta}}
\newcommand{\gn}{g^n}
\newcommand{\gtt}{g^t}
\newcommand{\GDt}{\Gamma^t_D}
\newcommand{\GDn}{\Gamma^n_D}
\newcommand{\GNt}{\Gamma^t_N}
\newcommand{\GNn}{\Gamma^n_N}
\newcommand{\RM}{\mathrm{RM}}

\usepackage{color}
\definecolor{deepblue}{rgb}{0,0,0.5}
\definecolor{deepred}{rgb}{0.6,0,0}
\definecolor{deepgreen}{rgb}{0,0.5,0}

\usepackage{amsmath,amssymb}
\usepackage{bm}
\usepackage{amsthm}
\usepackage{anyfontsize}
\usepackage{pgfbaselayers}
\usepackage{url}
\usepackage{graphicx}
\usepackage{array}
\usepackage{varwidth}
\usepackage{geometry}
\usepackage{ esint }
\definecolor{ao(english)}{rgb}{0.0, 0.5, 0.0}
\usepackage{listings}
\usepackage{fancyhdr}
\usepackage{enumitem}
\usepackage[
    backend=bibtex,
   bibstyle=ieee,
    citestyle=numeric-comp,
    natbib=true,
    sortlocale=en_GB,
    sorting = nyt,
    url=false,
    doi=true,
    arxiv=true,
    eprint=true
]
{biblatex}
\usepackage{caption}
\usepackage{mathtools}
\usepackage{appendix}
\usepackage{hyperref}
\usepackage{xcolor}
\hypersetup{
    colorlinks,
    linkcolor={blue!50!black},
    citecolor={blue!50!black},
    urlcolor={blue!80!black}
}
\usepackage{tikz-cd} 

\RequirePackage{algorithm}[0.1]

\RequirePackage[capitalize,nameinlink]{cleveref}[0.19]
\crefformat{equation}{\textup{#2(#1)#3}}
\crefrangeformat{equation}{\textup{#3(#1)#4--#5(#2)#6}}
\crefmultiformat{equation}{\textup{#2(#1)#3}}{ and \textup{#2(#1)#3}}
{, \textup{#2(#1)#3}}{, and \textup{#2(#1)#3}}
\crefrangemultiformat{equation}{\textup{#3(#1)#4--#5(#2)#6}}%
{ and \textup{#3(#1)#4--#5(#2)#6}}{, \textup{#3(#1)#4--#5(#2)#6}}{, and \textup{#3(#1)#4--#5(#2)#6}}

\Crefformat{equation}{#2Equation~\textup{(#1)}#3}
\Crefrangeformat{equation}{Equations~\textup{#3(#1)#4--#5(#2)#6}}
\Crefmultiformat{equation}{Equations~\textup{#2(#1)#3}}{ and \textup{#2(#1)#3}}
{, \textup{#2(#1)#3}}{, and \textup{#2(#1)#3}}
\Crefrangemultiformat{equation}{Equations~\textup{#3(#1)#4--#5(#2)#6}}%
{ and \textup{#3(#1)#4--#5(#2)#6}}{, \textup{#3(#1)#4--#5(#2)#6}}{, and \textup{#3(#1)#4--#5(#2)#6}}

\usepackage{algpseudocode}
\algnewcommand{\Initialize}[1]{
  \State \textbf{Initialize:}
  \Statex \hspace*{\algorithmicindent}\parbox[t]{.8\linewidth}{\raggedright #1}
}
\algnewcommand{\Indent}[2]{
  \State {#1}
  \vspace{-2mm}
  \Statex \hspace*{\algorithmicindent}\parbox[t]{.9\linewidth}{\raggedright #2}
}

\newtheorem{proposition}{Proposition}[section]
\newtheorem{theorem}{Theorem}[section]
\newtheorem{lemma}{Lemma}[section]
\newtheorem{definition}{Definition}[section]
\newtheorem{corollary}{Corollary}[section]
\newtheorem{remark}{Remark}[section]
\newtheorem{example}{Example}[section]
\numberwithin{equation}{section}

\title{{Numerical analysis of the SIMP model for the topology optimization problem of minimizing compliance in~linear~elasticity}}

\author{Ioannis P.~A.~Papadopoulos\thanks{Weierstrass Institute, DE, {\tt  papadopoulos@wias-berlin.de}.
\newline \hspace*{4.7mm} \textbf{Funding:} This work was completed with the support of the EPSRC grant EP/T022132/1 ``Spectral element methods for fractional differential equations, with applications in applied analysis and medical imaging", the Leverhulme Trust Research Project Grant RPG-2019-144 ``Constructive approximation theory on and inside algebraic curves and surfaces" and the Deutsche Forschungsgemeinschaft (DFG, German Research Foundation) under Germany’s Excellence Strategy – The Berlin Mathematics Research Center MATH + (EXC-2046/1, project ID: 390685689).  }}

\begin{document}
\maketitle
\date
\thispagestyle{empty}
\pagestyle{fancy}
\lhead{Ioannis Papadopoulos}

\begin{abstract}
We study the finite element approximation of the solid isotropic material with penalization  method  (SIMP) for the topology optimization   problem of minimizing  the compliance of a linearly elastic structure. To ensure the existence of a  local  minimizer to the infinite-dimensional problem, we consider two popular  regularization  methods: $W^{1,p}$-type  penalty methods  and density filtering. Previous results prove weak(-*) convergence in the space of the  material distribution to a  local minimizer  of the infinite-dimensional problem.  Notably, convergence was not guaranteed to \emph{all} the isolated local minimizers.  In this work, we show that, for every isolated local or global minimizer, there exists a sequence of finite element local minimizers that strongly converges to the minimizer in the appropriate space. As a by-product, this ensures that there exists a sequence of unfiltered discretized material distributions that does not exhibit checkerboarding.
\end{abstract}

\section{Introduction}

Topology optimization is an important mathematical technique extensively used in the initial stages of  engineering  design. The goal is to find the optimal design of a structure or device that minimizes an objective functional. In this work, we focus on the numerical analysis of the topology optimization  problem of minimizing  the compliance of linearly elastic materials. This is perhaps the most studied topology optimization problem, with roots that go back to the original works by Bends\o e and Kikuchi \cite{Bendsoe1988}. The topology optimization of elastic structures and its extensions have been used in numerous applications, for instance airplane wing design \cite{Aage2017}, cantilevers, and beams \cite{Bendsoe2004}. 

We consider the density formulation via the  solid isotropic material with penalization  method  (SIMP),  one of the most popular models for such problems.  Here the optimal  layout  of elastic material is encoded by a function, $\rho$, known as the material distribution. Regions where $\rho=1$ a.e.~indicate the presence of elastic material, whereas where $\rho=0$ a.e.~are areas where material is absent, as modelled by a material with negligible stiffness. The density approach results in an infinite-dimensional and nonconvex optimization problem with PDE, box, and inequality constraints. In order to obtain a well-posed problem, the infinite-dimensional model  is often  coupled with a \emph{ regularization  method} \cite{Bendsoe2004}. Irrespective of the choice of  the regularization  method, even with fixed model parameters, there may exist multiple (local) minimizers to the same density formulated problem, e.g.~\cite[Sec.~4.5 \& 4.6]{Papadopoulos2021a}.

Due to the nonlinear nature of the optimization problem, closed-form  expressions for the local minimizers are difficult to find. In practice, the  full optimization problem is discretized along with the elastic displacement $\vu$ and the material distribution $\rho$. In a discretize-then-optimize approach, the objective functional as well as the PDE, box, and inequality constraints are also discretized and one utilizes a finite-dimensional optimization algorithm to discover  discrete  local minima. Otherwise, in an optimize-then-discretize approach, one derives the first-order optimality conditions of the infinite-dimensional optimization problem and discretizes those, utilizing solvers to find  discrete  solutions. In either case, the finite element (FE) method is the most common choice for discretizing the SIMP optimization problem. Despite its abundant use, proofs that local FE minimizers (rigorously defined in \cref{def:FEmin}) converge to the local minimizers of the original infinite-dimensional problem, in a suitably strong sense, is often missing in the literature. Different classes of regularization methods require different assumptions and techniques for their analysis.  Another  difficulty is caused by the nonconvexity of the problem (in the general case). Literature that does prove convergence typically shows that a subsequence of the local FE minimizers converge to a local infinite-dimensional minimizer of the original problem. However, since the nonconvexity provides multiple candidates for the limits of these sequences, there is an inherent ambiguity  in  these statements.  In particular,  these arguments do not guarantee that \emph{every} isolated minimizer can be arbitrarily closely approximated by a sequence of local FE minimizers.

In the context of the topology optimization of Stokes flow \cite{Borrvall2003}, the issue of approximating multiple minimizers with the FE method was recently resolved \cite{Papadopoulos2021a, Papadopoulos2022a, Papadopoulos2022b, Papadopoulos2022c}. It was shown that, for every isolated local or global minimizer of the problem, there exists a sequence of FE solutions to the first-order optimality conditions that strongly converges to the infinite-dimensional minimizer as the mesh size tends to zero. We endeavour to prove a similar result here for any conforming FE spaces with an approximation property.  The different physics (Stokes vs.~linear elasticity) and the additional analysis of the regularization method are the two major obstacles that prevent the analysis in \cite{Papadopoulos2022a, Papadopoulos2022b} from directly translating to the SIMP model.  To reiterate, it is the discretization of the SIMP model for elasticity that is analyzed in this paper. 

We consider two  different  regularization  methods in this work: $W^{1,p}$-type  penalty methods  and density filtering (both described in  \cref{sec:model}).   Recall that $\rho$ and $\vu$ denote the material distribution and displacement of the elastic structure, respectively. Let $h$ denote the mesh size.  Our goal is to show that for every isolated local or global infinite-dimensional minimizer of the SIMP model, there exists a sequence of  local FE minimizers  that converges to the infinite-dimensional minimizer. In particular we show that: 
\begin{itemize}
\item With a $W^{1,p}$-type  penalty one has that $\vu_h \to \vu$ strongly in $H^1(\Omega)^d$, $\rho_h \to \rho$ strongly in $W^{1,p}(\Omega)$, and $\rho_h \to \rho$ strongly in $L^s(\Omega)$ for any $s \in [1,\infty)$.
\item Instead if one uses a (nonlinear) density filter, then we show that  $\vu_h \to \vu$ strongly in $H^1(\Omega)^d$ and $\rho_h \to \rho$ strongly in $L^s(\Omega)$ for any $s \in [1,\infty)$. Moreover, we prove that $\frho_h \to \frho$ strongly in $W^{1,q}(\Omega)$, provided $\frho_h \in W^{1,q}(\Omega)$, where $\frho$ and $\frho_h$ denote the filtered material distribution and the discretized filtered material distribution, respectively.
\end{itemize}  
All the stated convergence results are novel, particularly for nonlinear density filters \cite{Wadbro2015, Hagg2017}. 

 The main results of this work are stated in Theorems \labelcref{th:FEexistence}--\labelcref{th:FEexistence-filtering}. The SIMP model, $W^{1,p}$-type penalty methods, and density filtering are introduced in \cref{sec:model}, \cref{sec:w1p}, and \cref{sec:density-filtering}, respectively. We discuss reformulating the generalized linear elasticity PDE constraint in variational form in \cref{sec:variational-form}, detail existence theorems in \cref{sec:existence}, and define isolated minimizers in \cref{sec:isolated}. We provide a literature review of previous FE results in \cref{sec:lit-review}  and discuss the implications of the convergence results in \cref{sec:implications}.  We establish the FE discretization in \cref{sec:fem} where we discuss the assumptions of the discretization. In \cref{sec:lemmas} we give useful auxiliary results known in the literature. Their proofs are provided in \cref{sec:app:proofs} for completeness. In \cref{sec:regularization}, we focus on $W^{1,p}$-type  penalty methods  and prove \cref{th:FEexistence}. We consider density filtering in \cref{sec:filtering} and prove Theorems \labelcref{th:FEexistence-filtering2} and \labelcref{th:FEexistence-filtering}. Finally, in \cref{sec:conclusions}, we give our conclusions and detail potential future extensions.

\section{The SIMP model}
\label{sec:model}

Consider a bounded and open Lipschitz domain $\Omega \subset \mathbb{R}^d$, $d \in \{2,3\}$. Let $W^{s,p}(\Omega)$, $s \in (0,\infty)$, $p \in [1,\infty]$, and $L^q(\Omega)$, $q \in [1,\infty]$, denote the standard Sobolev and Lebesgue spaces, respectively \cite{Adams2003}. Define $H^s(\Omega) \coloneqq W^{s,2}(\Omega)$. Let $\Gamma \subseteq \partial \Omega$ be a subset of the boundary with nonzero Hausdorff measure $\mathcal{H}^{d-1}(\Gamma)>0$. Then we define:
\begin{align}
H^1_{{\Gamma}}(\Omega)^d &:= \{\vect{v} \in H^1(\Omega)^d: \vect{v}|_{\Gamma} = \vect{0} \},\label{def:fs3}
\end{align}
where $|_\Gamma$ is the standard boundary trace operator $|_\Gamma : W^{1,p}(\Omega) \to W^{1-\frac{1}{p}, p}(\partial \Omega)$ \cite{Gagliardo1957}.
 
The topology optimization problem is to find $\vu :  \Omega  \to \mathbb{R}^d$, $\vu \in H^1(\Omega)^d$  and $\rho :  \Omega  \to \mathbb{R}$, $\rho \in L^\infty(\Omega)$,  that  minimize 
\begin{align}
J(\vect{u},\rho)  \coloneqq l(\vu) + \mathcal{R}(\rho) \label{complianceopt}
\end{align}
where
\begin{align}
l(\vu) \coloneqq \int_{\Omega} \vf \cdot \vu \, \dx + \int_{\GNt} \gtt  \vu_t  \, \ds + \int_{\GNn} \gn  \vu_n \,  \ds, \label{complianceopt2}
\end{align}
subject to the  generalized  linear elasticity PDE constraint,  cf.~\cite[Sec.~2]{Ainsworth2022} and \cite{Bauer2016}, 
\begin{align}
\begin{split}
\begin{cases}
-\divv{\tens{S}} = \vf & \text{in} \; \Omega,\\
\tens{S} = k(F(\rho))\left[2 \mu  \tens{D}(\vect{u}) + \lambda \divv{\vu} \tens{I} \right]& \text{in} \; \Omega,\\ 
 \vu_t = 0  & \text{on} \; \GDt,\\
 \vu_n = 0  & \text{on} \; \GDn,\\
 (\tens{S}\vect{\hat{n}})_t   = \gtt  & \text{on} \; \GNt,\\
 (\tens{S} \vect{\hat{n}})_n   = \gn  & \text{on} \; \GNn,
\end{cases}
\end{split}
\label{eq:elasticitypde}
\end{align}
as well as the following box and inequality constraints on $\rho \in L^\infty(\Omega)$:
\begin{align}
0 \leq \rho \leq 1 \;\; \text{a.e.~in} \;\; \Omega \quad \text{and} \quad \int_\Omega F(\rho) \, \dx \leq \gamma |\Omega|. \label{eq:rhoconstraints}
\end{align}
The state $\vu$ denotes the  elastic  displacement of the structure and $\tens{S}$ denotes the stress tensor. The body force $\vf \in L^2(\Omega)^d$ and traction forces $\gtt \in L^2(\GNt)$ and $\gn \in L^2(\GNn)$ are known. Here  $\vect{\hat{n}}$ denotes the unit outward normal, $\vv_n \coloneqq \vv|_{\partial \Omega} \cdot \vect{\hat{n}}$, and $\vv_t \coloneqq \vv|_{\partial \Omega} - \vv_n \vect{\hat{n}}$, whenever $\vv_n$ and $\vv_t$ are well-defined.  $\GNn , \GDn \subset \partial \Omega$ (respectively $\GNt , \GDt \subset \partial \Omega$) are known disjoint boundaries on $\partial \Omega$ such that $\GNn \cup \GDn = \partial \Omega$ (respectively $\GNt \cup \GDt = \partial \Omega$). $\mu$ and $\lambda$ are the Lam\'e coefficients, $\text{tr}(\cdot)$ is the matrix-trace operator, $\tens{I}$ is the $d \times d$ identity matrix, and 
\begin{align*}
\tens{D}(\vect{u}) = \frac{1}{2}( \nabla \vect{u} + \nabla \vect{u}^\top), \;\;\;
k(\rho) = \epsilon_{\text{SIMP}}+ (1-\epsilon_{\text{SIMP}}) \rho^{p_s},
\end{align*} 
where $0< \epsilon_{\text{SIMP}} \ll1$ and $p_s \geq 1$. A typical ``magic" choice resulting in crisp  material distributions  is $p_s =3$ \cite[Sec.~3]{Sigmund2013}. The constant $\gamma \in (0,1)$ is the fraction of the total volume of the domain that the material distribution can occupy. Here, $| \cdot |$ denotes the Lebesgue measure of a set.  $\mR : W^{1,p}(\Omega) \to \mathbb{R}$ is a functional that models a $W^{1,p}$-type penalty and $F : \{\eta \in L^\infty(\Omega) : 0 \leq \eta \leq 1 \; \text{a.e.}\} \to  W^{1,\infty}(\Omega)$ is a function that encodes density filtering.

\begin{remark}
We have provided the linear elasticity PDE constraint in \cref{eq:elasticitypde} with very general boundary conditions. One could instead reduce the boundary conditions to two boundaries $\Gamma_D$ and $\Gamma_N$, $\Gamma_D \cup \Gamma_N = \partial \Omega$, such that 
\begin{align}
\begin{split}
\begin{cases}
\vu|_{\Gamma_D} = \vect{0}  & \text{on} \; \Gamma_D,\\
(\tS \vect{\hat{n}})|_{\Gamma_N} =  \vect{g}   & \text{on} \; \Gamma_N,
\label{eq:simpler-bcs}
\end{cases}
\end{split}
\end{align}
for some  $\vect{g} \in L^2(\Gamma_N)^d$.  All the subsequent results in this work still hold. However, we note that these simpler boundary conditions are insufficient to describe the physics for the MBB beam problem in \cref{ex:mbb} and many other common SIMP models.
\end{remark}

For a physical interpretation of the SIMP model we refer to Bends\o e and Sigmund \cite[Ch.~1]{Bendsoe2004}.  When $\rho \approx 1$, then $k(\rho) \approx 1$ indicating the presence of material, whereas where $\rho \approx 0$, then $k(\rho) \approx \epsilon_{\text{SIMP}}$ indicating void.  The optimal material distribution that minimizes the compliance is $\rho=1$ a.e.~in $\Omega$. However, due to the volume constraint restriction and the limitation $\rho \leq 1$ a.e., regions where $\rho < 1$ a.e.~become necessary.  The SIMP model penalizes intermediate regions where $0 < \rho < 1$ a.e.~because intermediate regions have a stiffness proportional to $k(F(\rho))$ for a volume that is proportional to $F(\rho)$. This makes such regions inefficient at minimizing compliance under a volume constraint. Thus, a local minimizer will favour larger regions where either $\rho=0$ or $\rho=1$ a.e., while significantly diminishing regions where $0 < \rho < 1$ a.e.~as $p_s \to \infty$. Hence, by raising $\rho$ to the power of $p_s$, values of $\rho$ between 0 and 1 are penalized.

The optimization problem \cref{complianceopt}--\cref{eq:rhoconstraints} is normally stated with $\mR(\rho) = 0$ and $F(\rho) = \rho$. With these choices \cref{complianceopt}--\cref{eq:rhoconstraints} is ill-posed in general, i.e.~it does not have minimizers in the continuous setting  \cite{Borrvall2001a}. In particular any attempt at a proof of existence, via a direct method in the calculus of variations \cite{Dacorogna2014}, fails.   Naive  attempts at finding minimizers often yield checkerboard patterns of $\rho$  (as discussed in \cref{sec:checkerboard}).  Checkerboarding can be avoided, without modifications to the discretized optimization problem induced by \cref{complianceopt}--\cref{eq:rhoconstraints}, by particular choices of FE spaces for $\vu$ and $\rho$. However, the  discretized material distribution  will still be mesh dependent i.e.~as the mesh is refined, the beams of the  discretized material distribution  will become ever thinner, leading to  the homogenized limit.  There are several schemes employed by the topology optimization community to obtain physically reasonable  local minimizers  known as  regularization  methods \cite{Bendsoe2004}. An excellent review is given by Sigmund and Petersson \cite{Sigmund1998} as well as Borrvall \cite{Borrvall2001a}. We consider two  regularization  methods in this work: $W^{1,p}$-type  penalty methods  and density filtering.



\subsection{$W^{1,p}$-type penalty methods}
\label{sec:w1p}
We use the functional $\mR$ to model $W^{1,p}$-type  penalties  of $\rho$ which forces one to seek a  local minimizer with the regularity $\rho \in L^\infty(\Omega) \cap W^{1,p}(\Omega)$.
\begin{definition}[$W^{1,p}$-type penalty methods]
\label{def:w1p}
For $p \in (1,\infty)$, we say that the regularization  method is a \emph{$W^{1,p}$-type penalty} method if $\mR(\rho) = \frac{\epsilon}{p} \| \nabla \rho\|^p_{L^p(\Omega)} + m(\rho)$, where $m(\rho)$ is continuous in the strong topology of $L^p(\Omega)$  and models lower order terms not involving $\nabla \rho$. For instance $m(\rho) = \int_\Omega \rho(1-\rho) \dx$ (if $p\geq 2$). We refer to $\mR$ as a \emph{$W^{1,p}$-type penalty functional}.
\end{definition}
We define the material distribution space $\fr$ as
\begin{align}
\fr \coloneqq \{ \rho \in L^\infty(\Omega) \cap W^{1,p}(\Omega) :  0 \leq \rho \leq 1 \; \text{a.e.}, \; \int_\Omega \rho \, \dx \leq \gamma |\Omega|\}.
\end{align}

Examples of choices for $\mR$ include, for $\epsilon > 0$, $\beta > 0$,
\begin{alignat}{2}
\mR(\rho) &= \frac{\epsilon}{p} \| \nabla \rho \|^p_{L^p(\Omega)}, \;\; p \in (1,\infty), &&\quad  \text{($W^{1,p}$-penalty),}\\
\mR(\rho) &= \frac{\beta \epsilon}{2} \| \nabla \rho\|^2_{L^2(\Omega)} + \frac{\beta}{2\epsilon} \int_\Omega \rho(1-\rho) \, \dx && \quad \text{(Ginzburg--Landau penalty).} \label{eq:ginzburg-landau}
\end{alignat} 
If $p=1$ or $p=\infty$, one recovers total variation regularization \cite{Petersson1999b} and a Lipschitz continuity penalty method \cite{Petersson1998}, respectively, which are beyond the scope of this work.

Often the $W^{1,p}$-type  penalty  is added as an additional constraint rather than at the level of the objective functional. For instance, $\mR(\rho) = 0$ but one adds the constraint $\| \nabla \rho\|_{L^p(\Omega)} \leq M < \infty$ for a user-chosen $M > 0$. Although the implementation differs, the analysis is similar and we expect many of the results discussed in this work to continue to hold.

\begin{example}[MBB beam]
\label{ex:mbb}
To illustrate the SIMP  model  with a $W^{1,p}$-type penalty,  we provide an example of an MBB beam. Consider the domain $\Omega = (0,3) \times (0,1)$. Consider the boundaries:
\begin{align}
\begin{split}
\Gamma^n_{D_1} = \{ x = 0 \}, \;\; \Gamma^n_{D_2} = \{ y =0, \; 2.9 \leq x \leq 3\}, \;\; \Gamma^n_{N_1} = \{ y = 1,\, 0 \leq x \leq 0.1\}.
\end{split}
\end{align}
The Dirichlet and Neumann boundaries are 
\begin{align}
\begin{split}
\GDn = \Gamma^n_{D_1} \cup  \Gamma^n_{D_2}, \;
\GDt = \emptyset, \;
\GNn  = \partial \Omega \backslash \GDn, \;
\GNt = \partial \Omega.
\end{split}
\end{align}
The body and traction forces are $\vf = \vect{0}$, $\gtt = 0$, and 
\begin{align}
\gn(x,y) =
\begin{cases}
-10 & \text{if} \; (x,y) \in \Gamma^n_{N_1}, \\
0 & \text{otherwise}.
\end{cases}
\end{align}
These conditions describe a half-beam that is fixed horizontally on the $y$-axis and fixed vertically at its bottom right corner on the $x$-axis. There is a boundary force pushing vertically downwards at the top left corner, which represents the middle of the beam when the half-beam is mirrored. This is visualised in \cref{fig:mbbsetup}. 

Consider the volume fraction $\gamma = 0.535$, the Lam\'e coefficients $\mu = 75.38$ and $\lambda=64.62$, the SIMP parameters, $\epsilon_{\text{SIMP}} = 10^{-5}$ and $p_s = 3$. If one uses a Ginzburg--Landau penalty, it is possible to find two local minima as depicted in \cref{fig:mbbsolns}.
\begin{figure}[h!]
\centering
\includegraphics[width =0.6 \textwidth]{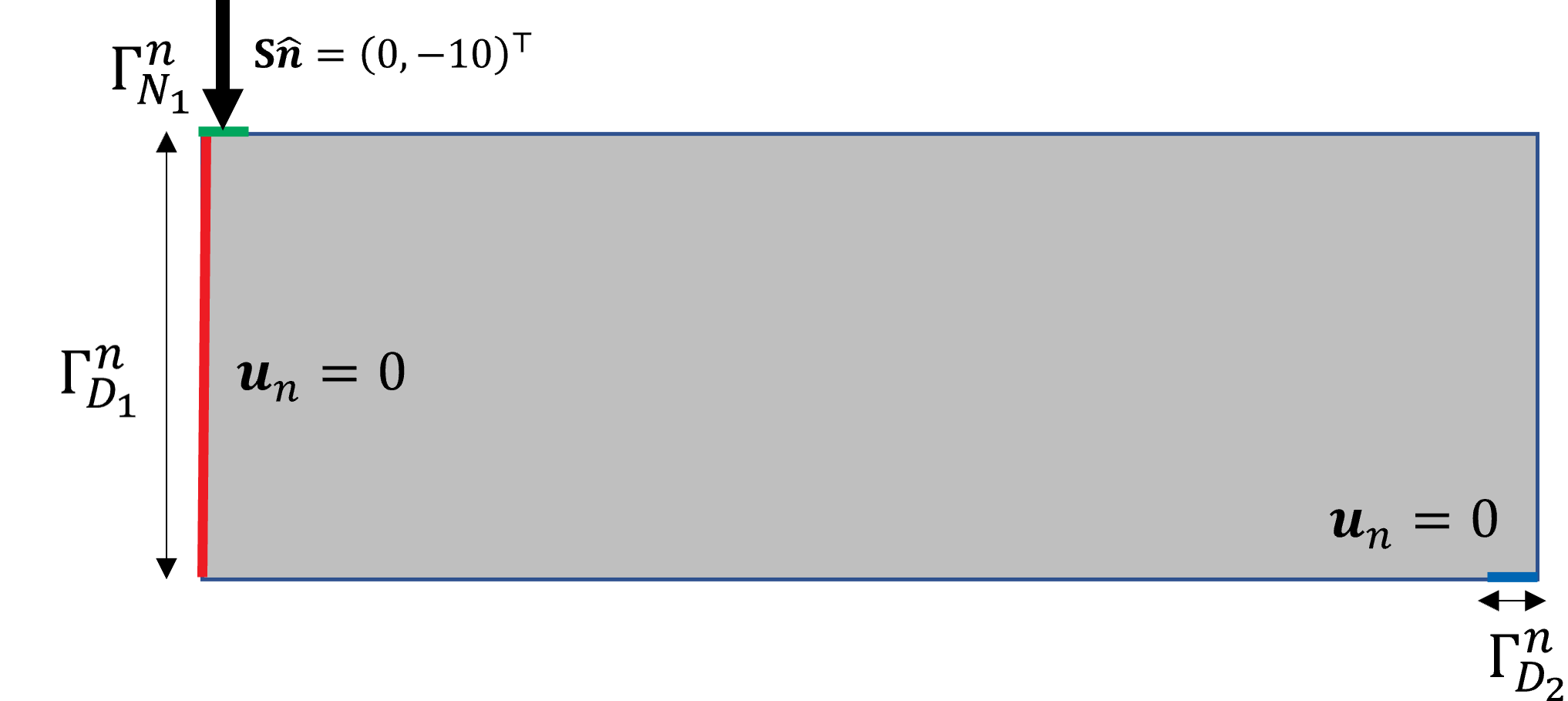}
\caption{Setup of the MBB beam. The remaining unlabeled boundary conditions are  $(\tS \vect{\hat{n}})_t = 0$  on $\Gamma^n_{D_1} \cup \Gamma^n_{D_2}$ and $\tS \vect{\hat{n}} = (0,0)^\top$ on $\partial \Omega \backslash (\Gamma^n_{D_1} \cup \Gamma^n_{D_2} \cup \Gamma^n_{N_1})$.}\label{fig:mbbsetup}
\end{figure}
\begin{figure}[h!]
\centering
\includegraphics[width =0.45 \textwidth]{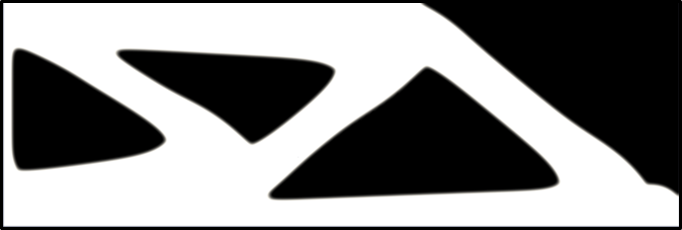}
\includegraphics[width =0.45 \textwidth]{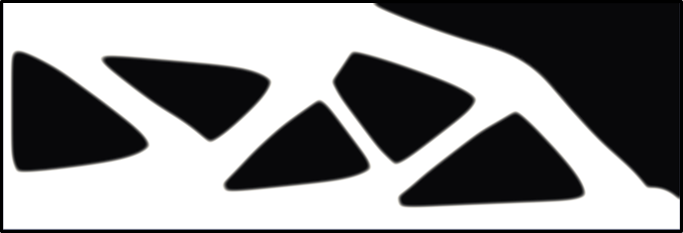}
\caption{The material distribution of two (local) minima of the MBB beam problem \cite[Sec.~4.6]{Papadopoulos2021a}. In white regions $\rho=1$ whereas in black regions $\rho = 0$. The  Ginzburg--Landau  parameters, as defined in \cref{eq:ginzburg-landau},  are $\epsilon = 1.90 \times 10^{-2}$, $\beta = 9 \times 10^{-3}$. Moreover, $\gamma = 0.535$, $\epsilon_{\text{SIMP}} = 10^{-5}$, $p_s = 3$, and the Lam\'e coefficients are $\mu = 75.38$ and $\lambda = 64.62$.}\label{fig:mbbsolns}
\end{figure}
\end{example}

\subsection{Density filtering}
\label{sec:density-filtering}
The function $F(\cdot)$ models a \emph{density filtering} of the problem \cite{Bourdin2001, Hagg2017, Wadbro2015}. Filtering averages $\rho$ in small neighbourhoods at every point. The first application in the context of topology optimization of compliance is attributed to Bruns and Tortorelli \cite{Bruns2001}; existence and FE convergence was first shown by Bourdin \cite{Bourdin2001}. We define the unfiltered material distribution space $\frF$ as
\begin{align}
\frF \coloneqq \{ \rho \in L^\infty(\Omega) :  0 \leq \rho \leq 1 \; \text{a.e.}, \; \int_\Omega F(\rho) \, \dx \leq \gamma |\Omega|\}
\end{align}
and the filtered material distribution space $\ffrF$
\begin{align}
\ffrF \coloneqq \{ \tilde{\eta} \in W^{1,\infty}(\Omega) : \exists \; \eta \in \frF \; \text{s.t.} \; \tilde{\eta} = F(\eta) \}. \label{def:ffr}
\end{align}
In the context of density filtering, the material distribution, prior to filtering, $\rho \in \frF$, is known as the \emph{unfiltered} material distribution and $\frho = F(\rho) \in \ffrF$ is known as the \emph{filtered} material distribution.

We make the following assumptions on the density filters:
\begin{enumerate}[label=({F}\arabic*)]
\item For any $\eta \in \frF$, then $\|F(\eta) \|_{W^{1,\infty}(\Omega)} \leq C < \infty$ with $C$ independent of $\eta$; \label{ass:filtering2}
\item For any $\eta \in \frF$, we have $0 \leq F(\eta) \leq 1$ a.e.~in $\Omega$; \label{ass:filtering-box}
\item Suppose there exists an $L^s$-bounded sequence, $s \in [1,\infty]$, such that $\eta_j \weak \eta$ weakly in $L^s(\Omega)$ ($\eta_j \weakstar \eta$ weakly-* in $L^\infty(\Omega)$ if $s = \infty$). Then there exists a subsequence (up to relabelling) such that $F(\eta_j) \to F(\eta)$ strongly in $L^s(\Omega)$. \label{ass:filtering}
\end{enumerate}
To prove convergence of the filtered material distribution in $W^{1,s}(\Omega)$,  for some $s \in (1,\infty)$, we require the following assumption:
\begin{enumerate}[label=({F}\arabic*)]
\setcounter{enumi}{3}
\item The space of filtered material distributions, $\ffrF$, is a norm-closed and convex subset of $W^{1,q}(\Omega)$ for some $q \in (1,\infty)$. \label{ass:filtering4}
\end{enumerate}

One class of density filters are known as linear density filters.
\begin{definition}
\label{def:linear-filter}
A density filter $F$ is a \emph{linear density filter} if the material distribution $\rho$ is convolved with a function that is not dependent on $\rho$, i.e.
\begin{align}
F(\rho)(x) = \int_\Omega f(x-y) \rho(y) \, \dy, 
\end{align}
where $f \in W^{1,\infty}(\mathbb{R}^d)$, $f \geq 0$ a.e.~and $\| f \|_{L^1(\mathbb{R}^d)} = 1$.
\end{definition}

The following proposition is a known result and we give a proof in \cref{sec:app:proofs} for completeness. 
\begin{proposition}
\label{prop:filter:ass}
Suppose that $F$ is a linear density filter. Then, \labelcref{ass:filtering2}--\labelcref{ass:filtering4} are satisfied.
\end{proposition}

We do not  analyze  \emph{sensitivity filtering} in this work \cite{Sigmund1994}.  Sensitivity filtering is cheap and effective and, under suitable conditions, may be understood as the identification of the Fr\'echet derivative of the objective functional with a gradient in a suitably chosen Hilbert space \cite[Sec.~5.2]{Allaire2021}, see also \cite{Evgrafov2019, Lazarov2011} for other studies. However, it does not fit under our framework and thus  it is  not studied here.

\subsection{Variational form} 
\label{sec:variational-form}
For a FE discretization of the linear elasticity PDE constraint, we must recast \cref{eq:elasticitypde} in variational form.  The most common choice, in the topology optimization literature, is the primal formulation and all known FE convergence results for the SIMP model only deal with this case. The primal formulation substitutes the explicit formula for the stress tensor $\tS$ into the first equation of \cref{eq:elasticitypde}. Thus the two PDE constraints reduce to one PDE constraint in terms of $\vu$ and $\rho$. This PDE constraint is then multiplied by a test function $\vv$ and an integration by parts is performed. The advantage of the primal formulation lies in its simplicity and cost-effectiveness, attributed to the reduced computational load from not discretizing the stress tensor, thus saving degrees of freedom.

By choosing not to eliminate the stress tensor, one may form a mixed formulation whereby one discretizes and solves for the stress tensor and elastic displacement independently. There are multiple mixed formulations found in the literature and we refer the interested reader to \cite[Ch.~VI.3]{Braess2007}. In this work we solely consider a primal formulation and leave the analysis of different mixed formulations for future work.

Let $H^1_D(\Omega)^d \coloneqq \{ \vv \in H^1(\Omega)^d :  \vv_t|_{\GDt} =0, \, \vv_n|_{\GDn} = 0  \}$ and denote the space of rigid body motions by $\RM^d$ \cite[Sec.~2.2]{Bauer2016}:
\begin{align}
\RM^d \coloneqq \{ \tens{R} \vect{x} + \vect{a} : \vect{a}\ \in \mathbb{R}^d, \, \tens{R} \in \mathrm{skew}_d\},
\end{align}
where $\mathrm{skew}_d$ denotes the vector space of constant skew-symmetric $d \times d$ matrices. For instance, $\RM^2 = \text{span}\{[1,0]^\top, [0,1]^\top, [-y,x]^\top \}$. Define the space of displacements by $\Hu$ where
\begin{align}
\Hu \coloneqq \{ \vv \in H^1_D(\Omega)^d : (\nabla \vv,  \tens{R})_{L^2(\Omega)} = 0 \;\forall \; \tens{R} \in \nabla (\RM^d \cap H^1_D(\Omega)^d)\},
\label{def:Hu}
\end{align}
 where, given two matrices of functions $\tens{A}, \tens{B}  \in L^2(\Omega)^{d \times d}$, then $(\tens{A},\tens{B})^2_{L^2(\Omega)} \coloneqq  \sum_{i,j=1}^d (\tens{A}_{ij}, \tens{B}_{ij})^2_{L^2(\Omega)}$. For instance, given the boundary conditions in \cref{eq:simpler-bcs} we have that $H^1_D(\Omega)^d \cap \RM^d = \{ \vect{0} \}$ and the definition of $\Hu$ in \cref{def:Hu} reduces to $\Hu = H^1_D(\Omega)^d$.  The additional condition $(\nabla \vv, \tens{R})_{L^2(\Omega)} =0$ plays a role if the boundary conditions of the elastic problem are not sufficient to uniquely determine the solution. 

 Recall the definition of $l(\vv)$ in \cref{complianceopt2} and let $\mH$ denote
\begin{align}
\mH \coloneqq 
\begin{cases}
\fr & \text{if $\mR(\rho) \neq 0$ and $F(\rho) = \rho$},\\
\frF  & \text{if $\mR(\rho) = 0$ and $F(\rho) \neq \rho$},\\
\fr \cap \frF & \text{if $\mR(\rho) \neq 0$ and $F(\rho) \neq \rho$}.
\end{cases}
\end{align}
 In the variational form, the primal formulation of the optimization problem reduces to,
\begin{align}
 \min_{(\vu, \rho) \in \Hu \times \mH}  J(\vu,\rho) \;\; \text{subject to} \;\;
a(\vu,\vv; \rho) = l(\vv)   \; \text{for all} \; \vv \in \Hu,  
\label{cto} \tag{SIMP}
\end{align}
where
\begin{align}
a(\vu,\vv; \rho) &\coloneqq \int_\Omega k(F(\rho)) \tE \vu : \tE \vv \, \dx,\\
\tE \vu : \tE \vv  &\coloneqq  2 \mu  \tens{D}(\vect{u}) : \tens{D}(\vv) + \lambda \divv{\vu} \,\divv{\vv}.
\end{align}

Note that $\rho$ only appears in the optimization problem \cref{complianceopt}--\cref{eq:rhoconstraints} in the box and volume constraints as well as the  $W^{1,p}$-type penalty term $\mR$.  Moreover, by assumption \labelcref{ass:filtering-box}, $0 \leq \rho \leq 1$ a.e.~implies that $0 \leq F(\rho) \leq 1$.  If $\mR(\rho) = 0$ (no $W^{1,p}$-type penalty, but a density filter is applied) then \cref{cto} may be recast as
\begin{align}
\min_{(\vu, \frho) \in \Hu \times \ffrF} \tilde{J}(\vu,\frho)  \;\; \text{subject to} \;\;
\tilde{a}(\vu,\vv; \frho) = l(\vv)    \; \text{for all} \; \vv \in \Hu,
\label{fcto} \tag{F-SIMP}
\end{align}
 where $\tilde{J}(\vu,\frho) =  J(\vu,\rho)$ and $\tilde{a}(\vu,\vv; \frho) = a(\vu,\vv; \rho)$.

\subsection{Existence of local minimizers}
\label{sec:existence}
In this subsection we provide existence theorems for \cref{cto}.

\begin{proposition}[Existence \& uniqueness with a fixed $\rho$]
\label{prop:uniqueness}
Suppose that  $\rho \in \mH$  is fixed, $l(\vv) = 0$ for all $\vv \in H^1_{ D }(\Omega)^d \cap \RM^d$, and \labelcref{ass:filtering-box} holds. Then, there exists a unique $\vu \in \Hu$ that satisfies the PDE constraint in \cref{cto}. 
\end{proposition}
\begin{proof}
The result is standard and follows as a direct consequence of Korn's inequality with tangential or normal boundary conditions \cite[Th.~11]{Bauer2016} and the Lax--Milgram Theorem \cite[Ch.~6.2, Th.~1]{Evans2010}. See also \cite[Ch.~11]{Brenner2008}.
\end{proof}

An essential result utilized in this work is the existence of a  local   minimizer with either a  $W^{1,p}$-penalty method  or a density filter.

\begin{theorem}[Existence of a local minimizer to \cref{cto}]
\label{prop:existence}
Suppose that the conditions of \cref{prop:uniqueness} hold. Consider the two cases:
\hfill \break
\noindent \textbf{\emph{(Case 1).}} If a $W^{1,p}$-type penalty method is used, then there exists a local minimizer $(\vu, \rho) \in \Hu \times  \fr$ of \cref{cto}. 

\smallskip
\noindent \textbf{\emph{(Case 2).}} If a density filter is used and \labelcref{ass:filtering2}--\labelcref{ass:filtering} hold, then there exists a local minimizer $(\vu, \rho) \in \Hu \times  \frF$ of \cref{cto}. 
\end{theorem}
\begin{proof}
The proof of (Case 1) may be found in \cite[Ch.~5.2.2 ]{Bendsoe2004} and the proof of (Case 2) is found in \cite[Sec.~3]{Bourdin2001}.
\end{proof}

\subsection{Isolated local minimizers}
\label{sec:isolated}

A key assumption for the FE convergence results in this work is that the local and global minimizers in the limit are \emph{isolated}. Consider an element $u \in U$ where $U$ is a Banach space. We denote
the closed ball of radius $r$ by $B_{r,U}(u)$:
\begin{align}
B_{r,U}(u) \coloneqq \{v \in U : \|u-v\|_U \leq r \}.
\end{align}
This is extended to tuples of functions as follows: consider $u_i \in U_i$, for $i=1,\dots,n$,	 where $U_i$ are Banach spaces.  We define $B_{r, U_1 \times \cdots \times U_n}(u_1,\dots,u_n)$ as  
\begin{align}
\begin{split}
&B_{r, U_1 \times \cdots \times U_n}(u_1,\dots,u_n)\\
&\coloneqq \{ (v_1,\dots,v_n) \in U_1 \times \cdots \times U_n: \sum_{i=1}^n \|v_i-u_i\|_{U_i} \leq r\}.
\end{split}
\end{align}

\begin{definition}[Isolated local minimizer of \cref{cto}]
We say that a local minimizer $(\vu, \rho) \in \Hu \times \mH$ of \cref{cto} is \emph{isolated} if there exists an $r>0$ such that $(\vu, \rho)$ is the unique local minimizer in $B_{r, H^1(\Omega) \times Y}(\vu,\rho) \cap (\Hu \times \mH)$. Here 
\begin{align}
Y = 
\begin{cases}
W^{1,p}(\Omega) & \text{if a $W^{1,p}$-type penalty is used,}\\
L^2(\Omega) & \text{otherwise.}
\end{cases}
\end{align}
We define the \emph{radius of the basin of attraction} as the largest value $r$ such that $(\vu,\rho)$ is the unique minimizer. We note that 
\begin{align*}
(\Hu \cap B_{r/2,H^1(\Omega)}(\vu)) &\times (\mH \cap B_{r/2,Y}(\rho)) \\
&\indent \subset B_{r, H^1(\Omega) \times Y}(\vu,\rho) \cap (\Hu \times \mH )
\end{align*}
and hence  $(\vu,\rho)$ is also the unique minimizer in $(\Hu \cap B_{r/2,H^1(\Omega)}(\vu)) \times (\mH \cap B_{r/2,Y}(\rho))$. 
\end{definition}

\begin{remark}
We make the assumption that $\rho \in \mH$ is isolated with respect to  the  $Y$-norm (as opposed to the $L^\infty$-norm). This is a stronger isolation assumption as discussed in \cite[Rem.~7]{Papadopoulos2022a}.  To our knowledge,  the validity of the $Y$-isolation in SIMP optimization problems, regularized by a $W^{1,p}$-type penalty method or density filtering, is an open problem. However, qualitatively and numerically, it appears to hold in many practical problems e.g.~in \cref{ex:mbb}.  We make this stronger isolation assumption as  continuous functions are not dense in $L^\infty(\Omega)$, but are dense in $Y$. This has implications in the assumptions for the FE spaces (in particular \labelcref{ass:dense}).
\end{remark}

\begin{remark}[Violating the isolation assumption]
The results in Theorems \labelcref{th:FEexistence}--\labelcref{th:FEexistence-filtering} heavily rely on the isolation assumption in order to deduce the limit of a FE minimizing sequence. If the isolation assumption was lifted, the results, at least as derived in this work, would no longer be valid. Compliance problems, in full generality, can support very interesting solution landscapes, e.g.~an infinite number of global minimizers, no classical minimizers or local minimizers that are not necessarily isolated \cite[Ch.~4.1.5]{Allaire2002b}. A study of such local minimizers is beyond the scope of this work.
\end{remark}

\subsection{Literature review}
\label{sec:lit-review}

The focus of this subsection is to conduct a literature review and contextualize our aims. We define a local FE minimizer in \cref{def:FEmin}. Our goal is to answer the following open problems in the numerical analysis:
\begin{enumerate}[label=({P}\arabic*)]
\item With the exception of cases where a unique global minimizer exists, prior results do not explicitly identify the specific local minimizer towards which the sequence of local FE minimizers converges. This ambiguity arises due to the nonconvex nature of the problem, which results in multiple potential candidates for the limits. We show that for every isolated local minimizer, there exists a sequence of local FE minimizers that strongly converges to it; \label{open:nonconvexity}
\item ($W^{1,p}$-type penalty). The strongest results currently found in the literature show that $\rho_h \to \rho$ strongly in $L^s(\Omega)$, for any $s \in [1,\infty)$. We prove the stronger result: $\rho_h \to \rho$ strongly in $W^{1,p}(\Omega)$; \label{open:regularization}
\item (Density filter). Previous results showed that the unfiltered material distribution $\rho_h \weakstar \rho$ weakly-* in $L^\infty(\Omega)$. We improve this result to $\rho_h \to \rho$ strongly $L^s(\Omega)$, for any $s \in [1,\infty)$; \label{open:filtering}
\item (Density filter). It is known that the discretized filtered material distribution $F_h(\rho_h) \to F(\rho)$ uniformly on $\Omega$ ($F$ and $F_h$ are defined in \cref{sec:model} and \cref{sec:fem}, respectively). We prove that there exists a sequence such that $F_h(\rho_h) \to F(\rho)$ strongly in $W^{1,q}(\Omega)$ for any $q \in [1,\infty)$. \label{open:filtering2}
\end{enumerate}
In both Theorems \labelcref{th:FEexistence} and \labelcref{th:FEexistence-filtering2} we tackle the issue of multiple local minimizers \labelcref{open:nonconvexity}. In Theorems \labelcref{th:FEexistence}, \labelcref{th:FEexistence-filtering2}, and  \labelcref{th:FEexistence-filtering} we prove \labelcref{open:regularization}, \labelcref{open:filtering}, and \labelcref{open:filtering2}, respectively.

Below we give a (perhaps non-exhaustive) list of the known FE convergence results. We emphasize that in all of the following results convergence is \emph{not} guaranteed to all isolated local minimizers (unless there only exists one local minimizer).

\underline{SIMP coupled with a $W^{1,p}$-type penalty method.}
\begin{itemize}
\itemsep=0pt
\item Petersson and Sigmund analyzed the primal formulation in the context of slope constraints \cite{Petersson1998}, which have a similar flavour to a $W^{1,\infty}$-type penalty \cite[Sec.~6.3]{Borrvall2001a}. Here the admissible set of material distributions is reduced to Lipschitz continuous functions and a pointwise bound is placed on $\partial_{x_i} \rho$ for $i=1,2$. They showed that a $(\mathrm{CG}_1)^2$ (continuous piecewise linear) discretization for $\vu$ and a $\mathrm{DG}_0$ (piecewise constant) discretization for $\rho$ converges to a local minimum of the problem. In particular $\vu_h \to \vu$ strongly in $H^1(\Omega)^2$, $\rho_h \to \rho$ uniformly in $\Omega$ and, therefore, $\rho_h \to \rho$ strongly in $L^\infty(\Omega)$. 

\item Petersson \cite{Petersson1999b} investigated a regularization in the flavour of a \\ $W^{1,1}(\Omega)$-penalty together with penalization terms such as $\epsilon^{-1} \int_\Omega \rho(1-\rho)\,\dx$. He considered a $(\mathrm{CG}_1)^2$ discretization for $\vu$ and a $\mathrm{DG}_0$ discretization for $\rho$. He showed that a sequence of local FE minimizers of the discretized \cref{cto}, $(\vu_h, \rho_h)$, converges to  a local minimizer  $(\vu,\rho)$ such that $(\vu_h, \rho_h) \to (\vu,\rho)$ strongly in $H^1(\Omega)^d \times L^s(\Omega)$, for any $s \in [1,\infty)$.

\item Talischi and Paulino studied a $W^{1,2}$-penalty \cite{Talischi2015}. They used a $(\mathrm{CG}_k)^2$ and $(\mathrm{CG}_k)$, $k\geq1$, for the discretization of $\vu$ and $\rho$, respectively. However, they allow for $\rho$ to be projected  onto  the space $\mathrm{DG}_0$ in the PDE constraint. They showed that there exists a sequence of local FE minimizers, $(\vu_h,\rho_h)$ converging to  a local minimizer  $(\vu,\rho)$ such that $\vu_h \to \vu$ strongly in $H^1(\Omega)^2$ and $\rho_h \weak \rho$ weakly in $H^1(\Omega)$. 
\end{itemize}

\underline{SIMP with density filtering.}
\begin{itemize}
\itemsep=0pt
\item In the context of linear density filtering, Bourdin showed that, for a $(\mathrm{CG}_1)^2$ discretization for $\vu$ and a $\mathrm{DG}_0$ discretization for $\rho$, the sequence of minimizers $(\vu_h,\rho_h)$ of the discretized \cref{cto} converges to  a local minimizer  $(\vu,\rho)$ such that $\vu_h \to \vu$ strongly in $H^1(\Omega)^d$, $F_h(\rho_h) \to F(\rho)$ uniformly on $\Omega$, and $\rho_h \weak \rho$ weakly in $L^s(\Omega)$ for any $s \in [1,\infty)$ \cite{Bourdin2001}.
\end{itemize}

\underline{Other results.}

\begin{itemize}
\itemsep=0pt
\item One of the first FE convergence results for the topology optimization of a linearly elastic material was shown by Petersson and Haslinger \cite{Petersson1998b}. They showed that a sequence of local FE minimizers, $(\vu_h, \rho_h)$, of a problem similar to the discretization of \cref{cto}, converge to a local minimizer $(\vu,\rho)$. In particular $\vu_h \to \vu$ strongly in $H^1(\Omega)^d$ and $\rho_h \weakstar \rho$ weakly-* in $L^\infty(\Omega)$.

\item  The problem reduces to the problem of the optimal variable thickness of sheets when $p_s=1$. Petersson analyzed such problems \cite{Petersson1999} and showed that, under a regularity and biaxiality assumption on the optimal stress field, the minimizer $\rho$ is unique. Moreover, for any conforming discretization of $\vu$ and a $\mathrm{DG}_0$ discretization for $\rho$, the minimizing sequence of unique FE minimizers converges to the unique minimizer $(\vu,\rho)$ such that $(\vu_h,\rho_h) \to (\vu,\rho)$ strongly in $H^1(\Omega)^2 \times L^s(\Omega)$ for any $s \in [1,\infty)$.

\item  Borrvall and Petersson \cite{Borrvall2001b} considered a problem with no filtering, $F(\rho) = \rho$, but with a regularization term $\mR(\rho)$ reminiscent of a density filter. For example $\mR(\rho) = g \circ \int_\Omega f(x-y) \rho(y) \, \dy$ where $g(\tilde{\rho}) = \int_\Omega (\tilde{\rho} - \underline{\rho}) (1-\tilde{\rho}) \, \dx$, $\underline{\rho} > 0$, and $\circ$ denotes the composition operator. They showed that, for any conforming discretization of $\vu$ and a $\mathrm{DG}_0$ discretization for $\rho$, there exists a sequence of local FE minimizers $\rho_h$ such that $\rho_h \weakstar \rho$ weakly-* in $L^\infty(\Omega)$ where $\rho$ is  a local minimizer  of the problem. This is subsequently strengthened to $\rho_h \to \rho$ strongly in  $L^s(\Omega_{\rho,b})$,  for any $s \in [1,\infty)$, where  $\Omega_{\rho, b} \coloneqq \{ {\bf x} \in \Omega : \rho({\bf x}) = 0 \; \text{or} \; \rho({\bf x}) = 1\}$,  if such a set  $\Omega_{\rho,b}$  exists and is measurable.

\item Greifenstein and Stingl \cite{Greifenstein2016} consider a material distribution that is partially  density  filtered and partially regularized with slope constraints  (as defined in \cite{Petersson1998}).  They show that there exists a sequence of FE  minimizers  that converge to a local  minimizer.  In particular $\vu_h \to \vu$ strongly in $H^1(\Omega)^d$, the unfiltered material distribution converges weakly-* in $L^\infty(\Omega)$, whereas the filtered and slope constrained material distribution both converge uniformly. For more details on how the hybrid formulation is designed and how $\vu$ depends on the material distribution, we refer the reader to \cite{Greifenstein2016}.
\end{itemize}

Although the above list might not be exhaustive, all works follow very similar arguments as Petersson and Sigmund \cite{Petersson1998} or Bourdin \cite{Bourdin2001} for penalty and density filtering techniques, respectively.

\subsection{Implications}
\label{sec:implications}

\subsubsection{Multiple local minima} 
The optimization problem  is nonconvex and there may exist multiple local minima to the same optimization problem. In \cite[Sec.~4.5 \& 4.6]{Papadopoulos2021a}, there are examples of cantilevers and MBB beams (see \cref{ex:mbb}), coupled with a Ginzburg--Landau  penalty  method, where two local minima are found for each example.  All previous results only showed that a sequence of FE minimizers converge to a local minimizer of the infinite-dimensional problem. Our novel results guarantee strong convergence to all the isolated local minimizers. This is of critical importance in practical implementations.

\subsubsection{Checkerboard}
\label{sec:checkerboard}
We use the term \emph{checkerboarding} to describe the phenomenon where the discretized material distribution wildly oscillates in value between neighbouring cells in the mesh and this behaviour persists as the mesh size tends to zero \cite[Sec.~2]{Sigmund1998}. Terminating the mesh refinement at any point is problematic as, without careful post-processing, the discretized material distribution suggests designs which are not manufacturable. This may occur in the discretized material distribution of the SIMP model when a low-order discretization is used without a regularization method. The causes of checkerboarding are due to artificial stiffness granted by low-order finite elements  \cite{Diaz1995}  which in turn violate a discrete inf-sup condition of the discretized problem \cite{Jog1996}. Suitable regularization methods simultaneously guarantee the existence of a local minimizer to the SIMP model and prevent checkerboard patterns from persisting as $h \to 0$. 

Previous results in literature proved that a density filter ensures the existence of a sequence of unfiltered discretized material distributions such that $\rho_h \weakstar \rho$ weakly-* in $L^\infty(\Omega)$ where $\rho$ is an  unfiltered local material distribution minimizer of the infinite-dimensional problem. However, weak-* convergence is not sufficient to prevent a checkerboard pattern from forming and persisting. For instance, consider the one-dimensional checkerboard sequence $\phi_n \in L^\infty(0,1)$, $n \in \mathbb{N}$,
\begin{align}
\phi_n(x) =
\begin{cases}
1 & \text{if} \; j/n < x < (j+1/2)/n,\\
-1 & \text{if} \; (j+1/2)/n < x < (j+1/2)/n,\\
\end{cases}
\end{align}
where $j \in \{ 0, 1, \dots, n-1\}$. Then $\phi_n \weakstar 0$ weakly-* in $L^\infty(\Omega)$ \cite[Sec.~8.7]{Evans2010}. However $\phi_n \not\to 0$ strongly in $L^s(\Omega)$ for any $s \in [1,\infty]$ since $\| \phi_n \|_{L^s(\Omega)} = 1$ for all $n \in \mathbb{N}$. Thus, although weak-* convergence is not sufficient, by proving that there exists a  subsequence  (up to relabelling) where $\rho_h \to \rho$ strongly in $L^s(\Omega)$, for any $s \in [1,\infty)$, we guarantee that a checkerboard cannot persist as $h \to 0$ in both two and three dimensions \cite[Sec.~2]{Sigmund1998}. Although this fact was qualitatively observed in two dimensions \cite[Fig.~2 \& 3]{Bourdin2001}, it was not proven for the general class of density filters and FE methods we consider here. The unfiltered material distribution has sharper boundaries than the filtered material distribution and may be useful for inferring the eventual manufactured design.

\section{Finite element approximation}
\label{sec:fem}

 For the remainder of this work we assume that $\Omega \subset \mathbb{R}^d$ is a polygonal domain in two dimensions or a polyhedral  Lipschitz domain in three dimensions. Moreover, we assume that there exists a (local or global) isolated minimizer of \cref{cto} and fix one such \emph{local or global isolated minimizer}  $(\vect{u},\rho) \in \Hu \times \mH$ of \cref{cto}. 
 
Consider a family of quasi-uniform, non-degenerate,  and simplicial  meshes $(\mathcal{T}_h)$ as $h \to 0$ \cite[Def.~4.4.13]{Brenner2008}  such that $\cup_{K \in \mathcal{T}_h} K =   \bar \Omega$  for every $h$.  We choose conforming FE discretizations $\vX \subset H^1(\Omega)^d$, $\frh \subset \fr$, $\frFh \subset \frF$. To clarify, a FE discretization is conforming if the discretized space is a subspace of the infinite-dimensional space that it is discretizing.

In general, it will not be possible to represent the body and traction forces $\vf$, $\gn$, $\gtt$ exactly in the displacement FE space. Hence, for each $h$, we  instead consider $\vf_h$, $\gn_h$, $\gtt_h$ (which can be represented). We define the discretized functional:
\begin{align}
l_h(\vu_h) &\coloneqq  \int_{\Omega} \vf_h \cdot \vu_h \, \dx + \int_{\GNt} \gtt_h  (\vu_h)_t \,  \ds + \int_{\GNn} \gn_h  (\vu_h)_n  \, \ds.\label{complianceopt-h2}
\end{align}
We henceforth assume that:
\begin{enumerate}[label=({D}\arabic*)]
\item $\vf_h \to \vf$ strongly in $L^2(\Omega)^d$, $\gn_h \to \gn$ strongly in $L^2(\GNn)$, and $\gtt_h \to \gtt$ strongly in $L^2(\GNt)$. \label{ass:fh}
\end{enumerate}
Moreover,
\begin{enumerate}[label=({D}\arabic*)]
\setcounter{enumi}{1}
 \item The FE spaces are dense in their respective function spaces, i.e. for any $\vv \in H^1(\Omega)^d$,  $\eta \in \fr$, and $\theta \in \frF$, \label{ass:dense}
\begin{align*}
\begin{split}
\lim_{h \to 0} \inf_{\vect{w}_h \in \vX}\|\vv-\vect{w}_h\|_{H^1(\Omega)} &= \lim_{h\to 0} \inf_{\zeta_h \in \frh} \|\eta - \zeta_h \|_{W^{1,p}(\Omega)}\\
& = \lim_{h\to 0} \inf_{\zeta_h \in \frFh} \|\theta - \zeta_h \|_{L^2(\Omega)} = 0.
\end{split}
\end{align*}
\item $l(\vv) = l_h(\vv) =  0$ for all $\vv \in H^1_{ D }(\Omega)^d \cap \RM^d$. \label{ass:rigid-body}
\end{enumerate}
We denote the restriction of $\vX$ to the boundary conditions as:
\begin{align}
\vXD \coloneqq \vX \cap \Hu.
\end{align}

 \subsection{FE assumptions on the density filter} 
 Let $\mathcal{P}_k(\mathcal{T}_h) \subset W^{1,\infty}(\Omega)$, $k \geq 0$ denote the space of continuous piecewise polynomials of degree $k$ which are polynomials when restricted to each cell $K \in \mathcal{T}_h$.  In general, the filtered discretized material distribution $F(\rho_h) \not \in \mathcal{P}_k(\mathcal{T}_h)$. Hence, we are required to project $F(\rho_h)$ into the space $\mathcal{P}_k(\mathcal{T}_h)$. We denote the projection of $F(\rho_h)$  onto  $\mathcal{P}_k(\mathcal{T}_h)$ as $F_h(\rho_h) = \Pi_h F(\rho_h) \in \mathcal{P}_k(\mathcal{T}_h)$.  Here $\Pi_h$ denotes the projection operator cf.~\cite[Def.~2.3.9]{Brenner2008}.  We assume that:
\begin{enumerate}[label=({F}\arabic*)]
\setcounter{enumi}{4}
\item  For any $\eta \in  \frF$,  then $\|F_h(\eta) \|_{W^{s,q}(\Omega)} \leq C \|F(\eta) \|_{W^{s,q}(\Omega)}$, $s \in \{0,1\}$, $q \in [1,\infty]$, with $C$ independent of $h$; \label{ass:filtering3} 
\item The projection is a linear operator, i.e.~for any $u, v \in L^s(\Omega)$, $s\in [1,\infty]$, $a,b \in \mathbb{R}$, $\Pi_h(au+bv) = a\Pi_h u + b\Pi_h v$; \label{ass:filtering-projection2}
\item  For any $\eta \in  \frF$,  $\| F(\eta) - F_h(\eta)\|_{L^s(\Omega)} \to 0$ for any $s \in [1,\infty]$.   \label{ass:filtering-projection} %
\end{enumerate}
The interpolant  operator, as defined in \cite[Def.~3.3.9]{Brenner2008},  would satisfy the assumptions \labelcref{ass:filtering3}--\labelcref{ass:filtering-projection}, cf.~\cite[Th.~4.4.20]{Brenner2008} and \cite[Prop.~3.3.4]{Brenner2008}. 
Recall the definition of $\ffrF$ in \cref{def:ffr}. We define the FE space $\ffrFh$ as
\begin{align}
\ffrFh \coloneqq \{ \feta_h \in \frFh : \exists \; \eta_h \in \frFh \; \text{such that} \; \feta_h = F_h(\eta_h) \}. 
\end{align}
We make the following density assumption:
\begin{enumerate}[label=({D}\arabic*)]
\setcounter{enumi}{3}
\item The FE space $\ffrFh$ is dense in $\ffrF$, i.e.~for any $\feta \in \ffrF$, for some $q \in (1,\infty)$, \label{ass:dense:filtered}
\begin{align}
\lim_{h\to 0} \inf_{\tilde{\zeta}_h \in \ffrFh} \|\feta - \tilde{\zeta}_h  \|_{W^{1,q}(\Omega)} = 0.
\end{align}
\end{enumerate}

 \subsection{FE discretized optimization problem} 
The discretized compliance is denoted  by  $J_h$,
\begin{align}
J_h(\vu_h,\rho_h) &\coloneqq l_h(\vu_h) + \mathcal{R}(\rho_h),\label{complianceopt-h}
\end{align}
where $l_h$ is defined in \cref{complianceopt-h2}. We define the bilinear form
\begin{align}
a_h(\vu_h,\vv_h; \rho_h) \coloneqq \int_\Omega k(F_h(\rho_h)) \tE \vu_h : \tE \vv_h \, \dx.
\label{eq:weakformh}
\end{align}

\begin{definition}
\label{def:FEmin}
Consider the FE spaces $\vXD \subset H^1(\Omega)^d$ and $\mHh \subset \mH$. The \emph{FE discretized optimization problem} is to find $(\vect{u}_{h}, \rho_{h}) \in \vXD \times \mHh$ that satisfies,
\begin{align}
\begin{cases}
\min\limits_{(\vect{w}_h, \eta_h) \in \vXD \times \mHh}J_h(\vect{w}_h, \eta_h)\\
\text{subject to} \;\; a_h(\vect{w}_h,\vv_h; \eta_h) = l_h(\vv_h)\;\; \text{for all} \;\; \vv_h \in \vXD,
\end{cases}
 \label{ctoh} \tag{\text{SIMP$_h$}}
 \end{align}
where $J_h$, $l_h$, and $a_h(\cdot, \cdot; \cdot)$ are defined in \cref{complianceopt-h}, \cref{complianceopt-h2}, and \cref{eq:weakformh}, respectively. A local minimizer of \cref{ctoh} is called a \emph{local FE minimizer}.
\end{definition}

 \subsection{The main results} 

 The following three theorems are the core aims of this work.  In both Theorems \labelcref{th:FEexistence} and \labelcref{th:FEexistence-filtering2} we tackle the open issue of multiple local minimizers \labelcref{open:nonconvexity}. We prove that every isolated minimizer is arbitrarily closely approximated by the FE method. If a $W^{1,p}$-type  penalty  is used, in \cref{th:FEexistence}, we show that there exists a sequence of local FE minimizers satisfying $\vu_h \to \vu$ strongly in $H^1(\Omega)^d$ and $\rho_h \to \rho$ strongly in $W^{1,p}(\Omega)$. This resolves the open problem \labelcref{open:regularization}. If one chooses a density filter as a  regularization  method, then in \cref{th:FEexistence-filtering2}, we show that $\rho_h \to \rho$ strongly in $L^s(\Omega)$ for any $s\in[1,\infty)$,  resolving the open problem \labelcref{open:filtering}. Finally, we also prove that $F_h(\rho_h) \to F(\rho)$ strongly in $W^{1,q}(\Omega)$ as a result of \cref{th:FEexistence-filtering}. This observation resolves \labelcref{open:filtering2}.

\begin{theorem}[First main theorem: $W^{1,p}$-type penalty methods]
\label{th:FEexistence}
Suppose that $\mR$ is a $W^{1,p}$-type penalty functional for some $p \in (1,\infty)$ and $F(\rho) = \rho$. Consider the FE spaces $\vX \subset H^1(\Omega)^d$ and $\frh \subset \fr$. Suppose that \labelcref{ass:fh}--\labelcref{ass:rigid-body} hold. Fix an isolated local minimizer $(\vu, \rho) \in \Hu \times \fr$ of \cref{cto}. 
	
Then, as $h \to 0 $, there exists a sequence of local FE minimizers $(\vect{u}_{h}, \rho_{h}) \in \vXD \times \frh$ of \cref{ctoh} such that $\vu_h \to \vu$ strongly in $H^1(\Omega)^d$ and $\rho_h \to \rho$ strongly in $W^{1,p}(\Omega)$.
\end{theorem}

\begin{theorem}[Second main theorem: density filters]
\label{th:FEexistence-filtering2}
Suppose that a density filter is used that satisfies \labelcref{ass:filtering2}--\labelcref{ass:filtering-projection} and $\mR(\rho) = 0$. Consider the FE spaces $\vX \subset H^1(\Omega)^d$ and $\frFh \subset \frF$. Suppose that \labelcref{ass:fh}--\labelcref{ass:rigid-body} hold. Fix an isolated local minimizer $(\vu, \rho) \in \Hu \times \frF$ of \cref{cto}.
	
Then, as $h \to 0 $, there exists a sequence of local FE minimizers $(\vect{u}_{h}, \rho_{h}) \in \vXD \times \frFh$ of \cref{ctoh} such that $\vu_h \to \vu$ strongly in $H^1(\Omega)^d$ and $\rho_h \to \rho$ strongly in $L^s(\Omega)$ for any $s \in [1, \infty)$.
\end{theorem}

Recall that in the case of pure filtering \cref{cto} is equivalent to \cref{fcto}. Hence, if $(\vu,\frho)$ is the minimizer of \cref{fcto}, then $(\vu, \frho) = (\vu, F(\rho))$ where $(\vu, \rho)$ is the equivalent minimizer of \cref{cto}. 

\begin{theorem}[Third main theorem: density filters]
\label{th:FEexistence-filtering}
Suppose that the conditions of \cref{th:FEexistence-filtering2} hold, the space of filtered material distributions is $W^{1,q}$-conforming, $\ffrFh \subset (\ffrF \cap W^{1,q}(\Omega))$, for some $q \in (1,\infty)$, and the density assumption of \labelcref{ass:dense:filtered} is valid.

Assume that the minimizer $(\vu, \frho) = (\vu, F(\rho)) \in \Hu \times \ffrF$ is an isolated minimizer of \cref{fcto}. 

Then, there exists a subsequence (up to relabelling) such that
\begin{align}
\vu_h &\to \vu \; \text{strongly in} \; H^1(\Omega)^d,\\
F_h(\rho_h) &\to F(\rho) \; \text{strongly in} \; W^{1,q}(\Omega). 
\end{align}
\end{theorem}

%
%

\subsection{Auxiliary results}
\label{sec:lemmas}

The following lemmas are useful results and are used later in the convergence analysis. Their proofs are provided in \cref{sec:app:proofs} for completeness.

\begin{lemma}
\label{lem:density:strong}
Consider a sequence  $(\eta_j) \subset L^\infty(\Omega)$,   $0 \leq  \eta_j  \leq 1$ a.e.~such that  $\eta_j  \to \eta$  strongly in $L^q(\Omega)$ for some $q \in [1,\infty)$. Then  $\eta_j  \to \eta$  strongly in $L^s(\Omega)$ for any $s \in [1,\infty)$.
\end{lemma}

\begin{lemma}
\label{lem:filteringh}
Suppose that \labelcref{ass:filtering2}--\labelcref{ass:filtering-projection} hold. Consider any sequence $(\eta_h) \subset \frFh$ such that $\eta_h \weak \eta \in \frF$ weakly in $L^q(\Omega)$ ($\eta_h \weakstar \eta$ weakly-* in $L^\infty(\Omega)$ if $q=\infty$) for $q \in [1,\infty]$. Then $F_h(\eta_h)  \to F(\eta)$ strongly in $L^q(\Omega)$.
\end{lemma}

\begin{lemma}
\label{lem:weakpde}
Suppose that \labelcref{ass:fh}--\labelcref{ass:rigid-body} hold. Consider the following two cases:

\smallskip
\noindent \textbf{\emph{(Case 1:  $W^{1,p}$-type penalty).}} Suppose that $\mR$ is a $W^{1,p}$-type penalty functional, and $F(\rho) = \rho$. Consider any sequence $(\hat{\vu}_h, \hat{\rho}_h) \in \vXD \times \frh$ satisfying the PDE constraint in \cref{ctoh} such that $\hat{\vu}_h \weak \hat{\vu} \in \Hu$ weakly in $H^1(\Omega)^d$. Moreover, suppose that $\hat{\rho}_h \to \hat{\rho} \in \fr$ strongly in $L^q(\Omega)$ for some $q \in [2,\infty)$. Then the limit $(\hat{\vu}, \hat{\rho})$ satisfies the PDE constraint in \cref{cto}. 

\smallskip
\noindent \textbf{\emph{(Case 2:  density filtering).}} Suppose that a density filter is used that satisfies \labelcref{ass:filtering2}--\labelcref{ass:filtering-projection}, and $\mR(\rho) = 0$. Consider any sequence $(\hat{\vu}_h, \hat{\rho}_h) \in \vXD \times \frFh$ satisfying the PDE constraint in \cref{ctoh} such that $\hat{\vu}_h \weak \hat{\vu} \in \Hu$ weakly in $H^1(\Omega)^d$. Moreover, suppose that $\hat{\rho}_h \weak \hat{\rho}$ weakly in $L^q(\Omega)$ for some $q \in [2,\infty)$ ($\hat{\rho}_h \weakstar \hat{\rho}$ weakly-* in $L^\infty(\Omega)$ if $q = \infty$). Then the limit $(\hat{\vu}, \hat{\rho})$ satisfies the PDE constraint in \cref{cto}. 
\end{lemma}
Results similar to \cref{lem:weakpde} are found in the literature e.g.~\cite[Lem.~2.1]{Petersson1999b}.

\begin{lemma}[Existence of a strongly converging sequence $\hat{\vu}_h$ in $H^1(\Omega)^d$]
\label{lem:primal:strong:uh}
Suppose that \labelcref{ass:fh}--\labelcref{ass:rigid-body} hold. Consider a pair $(\vu, \rho) \in \Hu \times \mH$ that satisfies the PDE constraint in \cref{cto}. Consider the following two cases:

\smallskip
\noindent \textbf{\emph{(Case 1).}}  Suppose that $F(\rho) = \rho$. Consider any sequence such that $\hat{\rho}_h \to \rho$ strongly in $L^q(\Omega)$, for some $q \in [2,\infty]$, and the corresponding (unique) sequence of elastic displacements such that $(\hat{\vu}_h, \hat{\rho}_h)$ satisfies the PDE constraint in \cref{ctoh}. Then, $\hat{\vu}_h \to \vu$ strongly in $H^1(\Omega)^d$.

\smallskip
\noindent \textbf{\emph{(Case 2).}} Suppose that a density filter is used that satisfies \labelcref{ass:filtering2}--\labelcref{ass:filtering-projection}. Consider any sequence such that $F_h(\hat{\rho}_h) \to F(\rho)$  strongly in $L^q(\Omega)$, for some $q \in [2,\infty]$, and the corresponding (unique) sequence of elastic displacements such that $(\hat{\vu}_h, \hat{\rho}_h)$ satisfies the PDE constraint in \cref{ctoh}. Then, $\hat{\vu}_h \to \vu$ strongly in $H^1(\Omega)^d$.
\end{lemma}

\subsection{FE analysis: $W^{1,p}$-type  penalty methods }
\label{sec:regularization} 
In this subsection we assume that   a $W^{1,p}$-type  penalty is used, with functional $\mR$, as defined in \cref{def:w1p} and  no density filter is used, i.e.~$F(\eta) = F_h(\eta) = \eta$. 
Fix a local isolated minimizer  $(\vect{u},\rho) \in \Hu \times \fr$ of \cref{cto}.  Consider the finite-dimensional optimization problem: find $(\vu_h,\rho_h) \in \vXD \times \frh$ that minimizes 
\begin{align}
\begin{split}
\begin{cases}
\min\limits_{(\vv_h,\eta_h) \in (\vXD \cap B_{r/2,H^1(\Omega)}(\vu)) \times (\frh \cap B_{r/2,W^{1,p}(\Omega)}(\rho))} J_h(\vect{v}_h, \eta_h), \\
\text{subject to}\; a_h(\vu_h,\vv_h; \rho_h) = l_h(\vv_h) \;\; \text{for all} \; \vv_h \in \vXD.
\end{cases}
\end{split} \tag{\text{I-SIMP${}^p_{h}$}} \label{ictoh}
\end{align}
If there exists a pair $(\vu_h, \rho_h) \in (\vXD \cap B_{r/2,H^1(\Omega)}(\vu)) \times (\frh \cap B_{r/2,W^{1,p}(\Omega)}(\rho))$ such that $J_h(\vu_h, \rho_h) \leq J_h(\vv_h, \eta_h)$ for all $(\vv_h, \eta_h) \in (\vXD \cap B_{r/2,H^1(\Omega)}(\vu)) \times (\frh \cap B_{r/2,W^{1,p}(\Omega)}(\rho))$, then we call $(\vu_h, \rho_h)$ a \emph{global} FE minimizer.

In \cref{prop:FEexistence}, we prove the existence of finite-dimensional global minimizers to the discretized problem \cref{ictoh}.  Then, in \cref{prop:FEconvergence}, we prove weak convergence of the discretized global minimizers to the isolated infinite-dimensional local minimizer as $h \to 0$. Direct consequences of the weak convergence are derived in Corollaries \labelcref{cor:primal:reg:strong} and \labelcref{cor:primal:strong:uh}.
We derive a strong convergence result in \cref{prop:FE:regularization} and conclude this subsection by proving the result of \cref{th:FEexistence}.

\begin{proposition}
\label{prop:FEexistence}
Suppose that the conditions of \cref{th:FEexistence} hold. Then a  global FE minimizer $(\vu_h,\rho_h) \in \vXD \times \frh$  of \cref{ictoh} exists.
\end{proposition}
\begin{proof}
The set $(\vXD \cap B_{r/2,H^1(\Omega)}(\vu)) \times (\frh \cap B_{r/2, W^{1,p}(\Omega) }(\rho))$ is a finite-dimensional, closed, and bounded set. Moreover, for sufficiently small $h$ it is non-empty. Therefore, it is sequentially compact by the Heine--Borel theorem  \cite[Th.~11.18]{Fitzpatrick2009}. We say that $(\vv,\eta) \in \mS$ if $(\vv,\eta)$ satisfies the PDE constraint in \cref{ictoh}. By Korn's inequality \cite[Th.~11]{Bauer2016}, and the assumption that $\epsilon_{\text{SIMP}}>0$, then the bilinear form $a_h(\cdot,\cdot; \eta_h)$ is coercive for any $\eta_h \in \frh$. Moreover, the bilinear form is bounded,  \labelcref{ass:rigid-body} asserts that  $l_h(\vv) = 0$ for all $\vv \in  H^1(\Omega)^d \cap \RM^d$, and therefore, by the Lax--Milgram theorem \cite[Ch.~6.2, Th.~1]{Evans2010}, for every $\eta_h \in \frh$, there exists a unique $\vv_h \in \vXD$ such that $(\vv_h,\eta_h) \in \mS$. Moreover,  $\|\vv_h\|_{H^1(\Omega)}$ is bounded thanks to the boundedness of $\|\vf_h\|_{L^2(\Omega)}$, $\|\gn_h\|_{L^2(\Gamma^n_N)}$, and $\|\gtt_h\|_{L^2(\Gamma^t_N)}$.  Hence, $\mS \cap (\vXD \times \frh)$ is finite-dimensional, closed, and bounded and thus sequentially compact by the Heine--Borel theorem. As the spaces are Hausdorff, the intersection of two compact sets is compact and, hence, $\mS \cap  ((\vXD \cap B_{r/2,H^1(\Omega)}(\vu)) \times (\frh \cap B_{r/2, W^{1,p}(\Omega) }(\rho)))$ is sequentially compact.

By assumption the functional $J_h$ is continuous  in the strong topology of the finite-dimensional space $\vXD \times \frh$.  Hence $J_h$ obtains its infimum in $\mS \cap  ((\vXD \cap B_{r/2,H^1(\Omega)}(\vu)) \times (\frh \cap B_{r/2, W^{1,p}(\Omega) }(\rho)))$ and therefore, a  global FE minimizer  $(\vu_h, \rho_h)$ of \cref{ictoh} exists. 
\end{proof}

The following proposition is the first step in tackling the open problem \labelcref{open:nonconvexity}.

\begin{proposition}[Weak$\times$weak convergence of ($\vu_h, \rho_h)$ in $H^1(\Omega)^d \times W^{1,p}(\Omega)$]
\label{prop:FEconvergence}
Suppose that the conditions of \cref{th:FEexistence} hold.  Then there exists a subsequence of the  global FE minimizers  of \cref{ictoh}, that satisfies
\begin{align}
\vect{u}_h  &\weak \vect{u} \; \text{weakly in} \; H^1(\Omega)^d,\\
 \rho_h  &\weak \rho \; \text{weakly in} \;  W^{1,p}(\Omega).  
\end{align}
\end{proposition}
\begin{proof}
By a corollary of Kakutani's Theorem \cite[Th.~A.65]{Fonseca2006}, if a Banach space is reflexive then every norm-closed, bounded and convex subset of the Banach space is weakly compact and thus, by the Eberlein--\v{S}mulian theorem \cite[Th.~A.62]{Fonseca2006}, sequentially weakly compact. It can be checked that $\Hu \cap B_{r/2,H^1(\Omega)}(\vu)$ and $\fr \cap B_{r/2, W^{1,p}(\Omega) }(\rho)$ are norm-closed, bounded and convex subsets of the reflexive Banach spaces $H^1(\Omega)^d$ and $ W^{1,p}(\Omega) $, respectively. Therefore,  $\Hu \cap B_{r/2,H^1(\Omega)}(\vu)$ is weakly sequentially compact in $H^1(\Omega)^d$  and $\fr \cap B_{r/2, W^{1,p}(\Omega) }(\rho)$ is weakly sequentially compact in $ W^{1,p}(\Omega) $. 

Hence we extract subsequences (up to relabelling), $(\vu_h)$ and $(\rho_h)$ of the sequence generated by the global FE minimizers of \cref{ictoh} such that
\begin{align}
\vu_h &\weak \vu_0 \in \Hu \cap B_{r/2,H^1(\Omega)}(\vu) \; \text{weakly in} \; H^1(\Omega)^d,\\
\rho_h &\weak \rho_0 \in \fr \cap B_{r/2, W^{1,p}(\Omega) }(\rho) \; \text{weakly in} \;  W^{1,p}(\Omega) .
\end{align}
In the remainder of the proof, we show that the weak limit $(\vu_0, \rho_0)$ is in fact the isolated minimizer $(\vu,\rho)$. By the Rellich--Kondrachov theorem,  there exists a subsequence (up to relabelling) that satisfies
\begin{align}
\rho_h \to \rho_0 \;\; \text{strongly in} \;\; L^1(\Omega)
\label{eq:rellich}
\end{align} 
and thus $\rho_h \to \rho_0$ strongly in $L^s(\Omega)$ for any $s \in [1,\infty)$ by \cref{lem:density:strong}. Hence, we satisfy the requirement of \cref{lem:weakpde} and deduce that $(\vu_0, \rho_0)$ satisfies the PDE constraint in \cref{cto}.

The next step is to prove that $(\vu_0, \rho_0)$ is a minimizing pair. By assumption \labelcref{ass:dense}, there exists a sequence of FE functions $\hat{\rho}_h \in \frh$ that strongly converges to  $\rho$ in $ W^{1,p}(\Omega) $. Moreover, by \cref{lem:primal:strong:uh}, there exists a sequence ($\hat{\vu}_h)$ such that $(\hat{\vu}_h, \hat{\rho}_h)$ satisfies the PDE constraint in \cref{ictoh} and $\hat{\vu}_h \to \vu$ strongly in $H^1(\Omega)^d$. Thus 
\begin{align}
|J_h(\hat{\vu}_h,\hat{\rho}_h) -  J(\vu,\rho)| \leq |l_h(\hat{\vu}_h) - l(\vu)| + |\mR(\rho) - \mR(\hat{\rho}_h)|.
\end{align}
By \labelcref{ass:fh}, we have strong convergence of $\vf_h$ in $L^2(\Omega)$, $\gn_h$ in $L^2(\GNn)$, and $\gtt_h$ in $L^2(\GNt)$. Thus together with the strong convergence of $\hat{\rho}_h$ and $\hat{\vu}_h$ we see that
\begin{align*}
J_h(\hat{\vu}_h,\hat{\rho}_h) \to  J(\vu,\rho)  \;\; \text{as} \;\; h \to 0. 
\end{align*}
Furthermore, for sufficiently small $h >0$,
\begin{align*}
(\hat{\vu}_h, \hat{\rho}_h) \in (\vXD \cap B_{r/2,H^1(\Omega)}(\vu)) \times (\frh \cap B_{r/2, W^{1,p}(\Omega) }(\rho)).
\end{align*}
Therefore,
\begin{align}
J_h(\vu_h,\rho_h) \leq J_h(\hat{\vu}_h, \hat{\rho}_h).
\end{align}
By taking the limit as $h \to 0$ and utilizing the strong convergence of $\hat{\vu}_h$ and $\hat{\rho}_h$ to $\vu$ and $\rho$, respectively, we see that
\begin{align}
 \limsup_{h \to 0}  J_h(\vu_h, \rho_h) \leq J(\vu, \rho). \label{eq:femapprox1}
\end{align}
Recall that $J(\vu,\rho) = l(\vu) + \mR(\rho)$. Then by the linearity of $l(\vu)$ and \cref{def:w1p},  $J$ is weakly lower semicontinuous on $H^1(\Omega)^d \times  W^{1,p}(\Omega) $. Therefore,
\begin{align}
J(\vu_0,\rho_0) \leq \liminf_{h \to 0} J_h(\vu_h, \rho_h). \label{lowersemi}
\end{align}
$(\vu,\rho)$ is the unique 	local minimizer  of \cref{cto} in  $ (\Hu \cap B_{r/2,H^1(\Omega)}(\vu)) \times (\fr \cap B_{r/2, W^{1,p}(\Omega) }(\rho))$. Thus $J(\vu,\rho) \leq J(\vu_0,\rho_0)$. Hence, from \cref{eq:femapprox1} and \cref{lowersemi}, it follows that
\begin{align}
J(\vu_0,\rho_0) = J(\vu, \rho).
\end{align} 
Since $(\vu, \rho)$ is the unique local minimizer in the spaces we consider and $(\vu_0, \rho_0)$ satisfies the PDE constraint in \cref{cto}, we identify the limits $\vu_0$ and $\rho_0$ as $\vu$ and $\rho$, respectively, and state that $\vu_h \weak \vu$ weakly in $H^1(\Omega)^d$ and $\rho_h \weak \rho$ weakly in $ W^{1,p}(\Omega) $.
\end{proof}

\begin{corollary}
\label{cor:primal:reg:strong}
 Suppose that the conditions of \cref{th:FEexistence} hold.  Then there exists a subsequence of the  global FE minimizers  $(\vu_h,\rho_h)$ of \cref{ictoh} (up to relabelling) such that $\rho_h \to \rho$ strongly in $L^s(\Omega)$ for any $s \in [1,\infty)$.
\end{corollary}
\begin{proof}
By \cref{prop:FEconvergence}, $\rho_h \weak \rho$ weakly in $W^{1,p}(\Omega)$. Thus, for $p \geq 1$, the Rellich--Kondrachov theorem implies that there exists a subsequence such that $\rho_h \to \rho$ strongly in $L^1(\Omega)$. The result follows from \cref{lem:density:strong}.
\end{proof}

\begin{corollary}[Strong convergence of $\vu_h$ in $H^1(\Omega)^d$]
\label{cor:primal:strong:uh}
 Suppose that the conditions of \cref{th:FEexistence} hold.  There exists a subsequence of the   global FE minimizers  $(\vu_h,\rho_h)$ of \cref{ictoh} such that $\vu_h \to \vu$ strongly in $H^1(\Omega)^d$.
\end{corollary}

\begin{proof}
By \cref{cor:primal:reg:strong}, we have that $\rho_h \to \rho$ strongly in $L^2(\Omega)$. Hence, the result follows as a direct consequence of (Case 1) in \cref{lem:primal:strong:uh}.
\end{proof}

We now focus on proving the stronger results of \cref{th:FEexistence}. We show that there exists a subsequence (up to relabelling) such that $\rho_h \to \rho$ strongly in $W^{1,p}(\Omega)$.

\begin{proposition}[Strong convergence of $\rho_h$ in $W^{1,p}(\Omega)$, $p\in(1,\infty)$]
\label{prop:FE:regularization}
 Suppose that the conditions of \cref{th:FEexistence} hold.  Then, there exists a subsequence (up to relabelling) of   global FE minimizers  of \cref{ictoh}, that satisfies
\begin{align}
\rho_h \to \rho \; \text{strongly in} \; W^{1,p}(\Omega). 
\end{align}
\end{proposition}

\begin{proof}
By \labelcref{ass:dense}, there exists a sequence of functions $(\hat{\rho}_h)$ such that $\hat{\rho}_h \to \rho$ strongly in $W^{1,p}(\Omega)$.  Moreover, by (Case 1) in \cref{lem:primal:strong:uh} the corresponding sequence of elastic displacements $(\hat{\vu}_h)$ such that the pair $(\hat{\vu}_h,\hat{\rho}_h)$ satisfies the PDE constraint in \cref{ctoh} also satisfies $\hat{\vu}_h \to \vu$ strongly in $H^1(\Omega)^d$. Thanks to the strong convergence, for sufficiently small $h$, we have that $(\hat{\vu}_h,\hat{\rho}_h) \in (\vXD \cap B_{r/2,H^1(\Omega)}(\vu)) \times (\frh \cap B_{r/2,W^{1,p}(\Omega)}(\rho))$. Hence
\begin{align}
J(\vu, \rho) \leq J_h(\vu_h, \rho_h) \leq J_h(\hat{\vu}_h,\hat{\rho}_h).
\end{align}
In turn this implies that
\begin{align}
\begin{split}
&\frac{\epsilon}{p} \left| \|\nabla \rho\|^p_{L^p(\Omega)} - \|\nabla \rho_h\|^p_{L^p(\Omega)} \right|\\
& \indent \leq |J_h(\hat{\vu}_h,\hat{\rho}_h) - J(\vu, \rho)|
+ |l(\vu) - l_h(\vu_h)|
+ |m(\rho) - m(\rho_h)|.
\end{split} \label{eq:w1pstrong1}
\end{align}
Label the three terms on the right-hand side of \cref{eq:w1pstrong1} by (I), (II), and (III), respectively. Thanks to the strong convergence of $(\hat{\vu}_h,\hat{\rho}_h)$ in $H^1(\Omega)^d \times W^{1,p}(\Omega)$ then $\mathrm{(I)} \to 0$. Similarly due to \labelcref{ass:fh} and \cref{cor:primal:strong:uh}, we have that $\mathrm{(II)} \to 0$. Finally since, by assumption, $m(\cdot)$ is continuous in $L^p(\Omega)$,  \cref{cor:primal:reg:strong}  implies that $\mathrm{(III)} \to 0$. Therefore, $\|\nabla \rho_h\|^p_{L^p(\Omega)} \to \|\nabla \rho\|^p_{L^p(\Omega)}$ and consequently 
\begin{align}
\|\nabla \rho_h\|_{L^p(\Omega)} \to \|\nabla \rho\|_{L^p(\Omega)}.
\label{eq:w1pstrong2}
\end{align}
Whenever $p \in (1,\infty)$, $L^p(\Omega)$ is a uniformly convex Banach space. Hence, by the fact that $\rho_h \weak \rho$ weakly in $W^{1,p}(\Omega)$ (as shown in \cref{prop:FEconvergence}) and \cref{eq:w1pstrong2} holds, the Radon--Riesz property \cite[Th.~A.70]{Fonseca2006} implies that $\nabla \rho_h \to \nabla \rho$ strongly in $L^p(\Omega)$. Consequently $\rho_h \to \rho$ strongly in $W^{1,p}(\Omega)$. 
\end{proof}

\begin{corollary}
\label{cor:icto2cto:1}
 Suppose that the conditions of \cref{th:FEexistence} hold.  Then for sufficiently small $h$, there exists a subsequence of   global FE minimizers  $(\vu_h,\rho_h)$ of \cref{ictoh} that are also  local FE  minimizers  of \cref{ctoh}.
\end{corollary}
\begin{proof}
\cref{cor:primal:strong:uh} and \cref{prop:FE:regularization} imply that there exists a subsequence $(\vu_h,\rho_h)$ such that  $(\vu_h, \rho_h) \to (\vu,\rho)$ strongly in $H^1(\Omega)^d \times W^{1,p}(\Omega)$.  By the definition of strong convergence, there exists an $\bar h$ such that for all $h \leq \bar h$, $\| \vu - \vu_h\|_{H^1(\Omega)} + \| \rho - \rho_h\|_{W^{1,p}(\Omega)} \leq r/4$. Thus the basin of attraction constraint is not active and the subsequence of  global FE minimizers  of \cref{ictoh} are also  local FE minimizers  of \cref{ctoh}.
\end{proof}

We now have the sufficient ingredients to prove \cref{th:FEexistence}.

\begin{proof}[Proof of \cref{th:FEexistence}]
By \cref{prop:FEexistence}, there exists a sequence of  global FE minimizers  $(\vu_h, \rho_h)$ of \cref{ictoh} that satisfies $(\vu_h, \rho_h) \weak (\vu,\rho)$ weakly in $H^1(\Omega)^d \times W^{1,p}(\Omega)$. In \cref{cor:primal:strong:uh}, it was shown that there exists a subsequence (up to relabelling) such that $\vu_h \to \vu$ strongly in $H^1(\Omega)^d$. Then, in \cref{prop:FE:regularization}, we demonstrate that there exists a subsequence (up to relabelling) such that $\rho_h \to \rho$ strongly in $W^{1,p}(\Omega)$. Finally, in \cref{cor:icto2cto:1}, we conclude that a subsequence of the  global FE minimizers are also  local FE minimizers  of \cref{ctoh}. 
\end{proof}

\subsection{FE analysis: density filtering}
\label{sec:filtering}

In this subsection  we assume a density filter is applied, with a mapping $F$ that satisfies \labelcref{ass:filtering2}--\labelcref{ass:filtering-projection},  and that no penalty method is used, i.e.~$\mR \equiv 0$. Fix an isolated local minimizer  $(\vect{u},\rho) \in \Hu \times \frF$ of \cref{cto}. Consider the finite-dimensional optimization problem: find $(\vu_h,\rho_h) \in \vXD \times \frFh$ that minimizes 
\begin{align}
\begin{split}
\begin{cases}
\min\limits_{(\vv_h,\eta_h) \in (\vXD \cap B_{r/2,H^1(\Omega)}(\vu)) \times (\frFh \cap B_{r/2,L^2(\Omega)}(\rho))} J_h(\vect{v}_h, \eta_h), \\
\text{subject to}\; a_h(\vu_h,\vv_h; \rho_h) = l_h(\vv_h) \;\; \text{for any} \; \vv_h \in \vXD.
\end{cases}
\end{split} \tag{\text{I-SIMP${}^F_{h}$}} \label{ictoFh}
\end{align}

In \cref{prop:FEexistence:filtering}, we prove the existence of a finite-dimensional global minimizer to the discretized problem \cref{ictoFh}. Then, in \cref{prop:FEconvergence:filtering}, we prove weak convergence of the discretized minimizers to the isolated infinite-dimensional local minimizer as $h \to 0$. Direct consequences of the weak convergence are derived in Corollaries \labelcref{cor:weakstar}, \labelcref{cor:primal:reg:weak}, and \labelcref{cor:primal:strong:uh:filtering}.
We derive a strong convergence result in \cref{prop:FE:filtering-rho} for unfiltered material distributions which allows us to deduce the result of \cref{th:FEexistence-filtering2}. We conclude this subsection by proving the result of \cref{th:FEexistence-filtering}.

\begin{proposition}
\label{prop:FEexistence:filtering}
Suppose that the conditions of \cref{th:FEexistence-filtering2} hold. Then a  global FE minimizer  $(\vu_h,\rho_h) \in \vXD \times \frFh$ of \cref{ictoFh} exists.
\end{proposition}
\begin{proof}
The proof of this result follows the proof of \cref{prop:FEexistence} with some small modifications,  up to replacing the spaces  $\fr$ and $\frh$  with  $\frF$ and $\frFh$, respectively, and the optimization problem \cref{ictoh}  with  \cref{ictoFh}.
\end{proof}

\begin{proposition}[Weak$\times$weak convergence of ($\vu_h, \rho_h)$ in $H^1(\Omega)^d \times L^2(\Omega)$]
\label{prop:FEconvergence:filtering}
Suppose that the conditions of \cref{th:FEexistence-filtering2} hold. Then there exist subsequences (up to relabelling) of the  global FE minimizers  of \cref{ictoFh}, that satisfy
\begin{align}
\vect{u}_h  &\weak \vect{u} \; \text{weakly in} \; H^1(\Omega)^d,\\
\rho_h  &\weak \rho \; \text{weakly in} \; L^2(\Omega). 
\end{align}
\end{proposition}
\begin{proof}
The proof of this result follows the proof of \cref{prop:FEconvergence} with some small modifications,  up to replacing the spaces  $\fr$, $\frh$, and $W^{1,p}(\Omega)$  with  $\frF$, $\frFh$, and $L^2(\Omega)$, respectively, and the optimization problem \cref{ictoh}  with  \cref{ictoFh}. Moreover, we do not need to invoke the Rellich--Kondrachov theorem to deduce the strongly converging subsequence in \cref{eq:rellich} as, in this framework, the conditions of \cref{lem:weakpde} are already satisfied.
\end{proof}


\begin{corollary}
\label{cor:weakstar}
Suppose that the conditions of \cref{th:FEexistence-filtering2} hold. Then there exists a subsequence (up to relabelling) of the global FE minimizers $(\rho_h)$ of \cref{ictoFh} such that $\rho_h \weakstar \rho$ weakly-* in $L^\infty(\Omega)$ and consequently $\rho_h \weak \rho$ weakly in $L^s(\Omega)$ for all $s \in [1,\infty)$.
\end{corollary}
\begin{proof}
By the Banach--Alaoglu theorem,  the closed unit ball of the dual space of a normed vector space is compact in the weak-* topology and, therefore,  the closed unit ball of  $L^\infty(\Omega)$ is weak-* compact. Hence we find a subsequence (up to relabelling) such that $\rho_h \weakstar \rho_0 \in \frF \cap \{\eta: \|\rho-\eta\|_{L^\infty(\Omega)} \leq r/2\}$ weakly-* in $L^\infty(\Omega)$. By the uniqueness of the weak limit, we identify $\rho_0 = \rho$ a.e.~in $\Omega$ and, thus, we deduce that $\rho_h \weakstar \rho$ weakly-* in $L^\infty(\Omega)$,  i.e.~$\int_\Omega \rho_h \eta \, \mathrm{d}x \to \int_\Omega \rho \eta \, \mathrm{d}x$ for any $\eta \in L^1(\Omega)$. Since $L^s(\Omega) \subset L^1(\Omega)$, for any $s > 1$, we note that $\int_\Omega \rho_h \eta \, \mathrm{d}x \to \int_\Omega \rho \eta \, \dx$ for any $\eta \in L^s(\Omega)$, $s \geq 1$.  Hence, by definition,  $\rho_h \weak \rho$ weakly in $L^s(\Omega)$ for all $s \in [1,\infty)$. 
\end{proof}

\begin{corollary}
\label{cor:primal:reg:weak}
Suppose that the conditions of \cref{th:FEexistence-filtering2} hold. Let $(\rho_h)$ denote the sequence of global FE minimizers to \cref{ictoFh} such that $\rho_h \weakstar \rho$ weakly-* in $L^\infty(\Omega)$. Then $F_h(\rho_h) \to F(\rho)$ strongly in $L^\infty(\Omega)$.
\end{corollary}
\begin{proof}
The result is a direct consequence of \cref{lem:filteringh}. 
\end{proof}

\begin{corollary}[Strong convergence of $\vu_h$ in $H^1(\Omega)^d$]
\label{cor:primal:strong:uh:filtering}
Suppose that the conditions of \cref{th:FEexistence-filtering2} hold. Then there exists a subsequence (up to relabelling) of the global FE minimizers  $(\vu_h,\rho_h)$ of \cref{ictoFh} such that $\vu_h \to \vu$ strongly in $H^1(\Omega)^d$.
\end{corollary}
\begin{proof}
By \cref{cor:primal:reg:weak}, we have a subsequence such that $F_h(\rho_h) \to F(\rho)$ strongly in $L^\infty(\Omega)$. Hence, the result follows as a direct consequence of (Case 2) in \cref{lem:primal:strong:uh}.
\end{proof}

In the remainder of this subsection we tackle the open problems \labelcref{open:filtering} and \labelcref{open:filtering2}. In particular, we prove the stronger results of \cref{th:FEexistence-filtering2} and  \cref{th:FEexistence-filtering}  Recall that $(\vu, \rho)$ is the fixed local minimizer of \cref{cto}.

We first prove the following two lemmas. 
\begin{lemma}
\label{lem:eps-existence}
 Suppose that the conditions of \cref{th:FEexistence-filtering2} hold. Then  a local minimizer   exists to the following $\epsilon$-perturbed optimization problem:
\begin{align}
\begin{split}
\begin{cases}
\min\limits_{\vu_\epsilon,\rho_\epsilon} J_\epsilon(\vu_\epsilon,\rho_\epsilon) \coloneqq J(\vu_\epsilon,\rho_\epsilon) + \frac{\epsilon}{2} \|  \rho_\epsilon  \|^2_{L^2(\Omega)}, \\
\text{where} \;\; ( \vu_\epsilon,\rho_\epsilon) \in (\Hu \cap B_{r/2,H^1(\Omega)}(\vu)) \times (\frF \cap B_{r/2,L^2(\Omega)}(\rho))\\
\text{subject to} \;\; a(\vu_\epsilon, \vv; \rho_\epsilon) = l(\vv) \;\; \text{for all} \; \vv \in \Hu.
\end{cases}
\end{split} \tag{\text{I-SIMP${}^F_\epsilon$}} \label{ictoe}
\end{align}
\end{lemma}
\begin{proof}
Consider a minimizing sequence $(\vu_n, \rho_n) \in (\Hu \cap B_{r/2,H^1(\Omega)}(\vu)) \times (\frF \cap B_{r/2,L^2(\Omega)}(\rho))$. Since $\Hu \cap B_{r/2,H^1(\Omega)}(\vu)$ and $\frF \cap B_{r/2,L^2(\Omega)}(\rho)$ are norm-closed, bounded, and convex subsets of $H^1(\Omega)^d$ and $L^2(\Omega)$, respectively, then there exists a limit $(\vu_\epsilon, \rho_\epsilon) \in (\Hu \cap B_{r/2,H^1(\Omega)}(\vu)) \times (\frF \cap B_{r/2,L^2(\Omega)}(\rho))$ such that
\begin{align}
\vu_n &\weak \vu_\epsilon \;\; \text{weakly in} \;\; H^1(\Omega)^d,\\
\rho_n &\weak\rho_\epsilon \;\; \text{weakly in} \;\; L^2(\Omega).
\end{align}
By \cref{lem:weakpde}, $(\vu_\epsilon, \rho_\epsilon)$ satisfies the PDE constraint in \cref{ictoe}. Moreover, $J(\vu,\rho)$ is bounded and weak$\times$weak lower semicontinuous on $H^1(\Omega)^d \times L^2(\Omega)$. Thus $J_\epsilon(\vu_\epsilon, \rho_\epsilon) \leq \liminf_{n \to \infty} J_\epsilon(\vu_n, \rho_n)$. Hence, a local minimizer exists. 
\end{proof} 

\begin{lemma}
\label{lem:epsilon-weak}
 Suppose that the conditions of \cref{th:FEexistence-filtering2} hold.  Consider a sequence of local minimizers $(\vu_\epsilon, \rho_\epsilon)$ of \cref{ictoe}. Then, there exists a subsequence (up to relabelling) such that $\vu_\epsilon \weak \vu$ weakly in $H^1(\Omega)^d$ and $\rho_\epsilon \weak \rho$ weakly in $L^2(\Omega)$ as $\epsilon \to 0$. 
\end{lemma}
\begin{proof}
Since $\Hu \cap B_{r/2,H^1(\Omega)}(\vu)$ and $\frF \cap B_{r/2,L^2(\Omega)}(\rho)$ are norm-closed, bounded, and convex subsets of $H^1(\Omega)^d$ and $L^2(\Omega)$, respectively, then there exists a limit $(\vu_0, \rho_0) \in (\Hu \cap B_{r/2,H^1(\Omega)}(\vu)) \times (\frF \cap B_{r/2,L^2(\Omega)}(\rho))$ such that
\begin{align}
\vu_\epsilon &\weak \vu_0 \;\; \text{weakly in} \;\; H^1(\Omega)^d,\\
\rho_\epsilon &\weak\rho_0 \;\; \text{weakly in} \;\; L^2(\Omega).
\end{align}

By \cref{lem:weakpde}, $(\vu_0, \rho_0)$ satisfies the PDE constraint in \cref{cto}. 
 
Consider any sequence of functions $\rho_n \in \frF$ such that $\rho_n \to \rho$ strongly in $L^2(\Omega)$.  Then the corresponding sequence of elastic displacements $\vu_n \in \Hu$ such that the pair $(\vu_n, \rho_n)$ satisfies the PDE constraint in \cref{cto} also satisfies $\vu_n \to \vu$ strongly in $H^1(\Omega)^d$.  For any $\epsilon > 0$ and sufficiently large $n \in \mathbb{N}$,
\begin{align}
J_\epsilon(\vu_\epsilon, \rho_\epsilon) \leq J_\epsilon(\vu_n, \rho_n).
\end{align}
$J_\epsilon$ is weak$\times$weak lower semicontinuous on $H^1(\Omega)^d \times L^2(\Omega)$. Thus by taking the limits $\epsilon \to 0$ and $n \to \infty$ and using the weak convergence of $(\vu_\epsilon, \rho_\epsilon)$ and the strong convergence of $(\vu_n, \rho_n)$ we see that
\begin{align} 
J(\vu_0, \rho_0)  \leq \liminf_{\epsilon \to 0} J_\epsilon(\vu_\epsilon, \rho_\epsilon) \leq  J(\vu,\rho).
\end{align} 
Since $(\vu, \rho)$ is the unique minimizer of \cref{cto} in $(\Hu \cap B_{r/2,H^1(\Omega)}(\vu)) \times (\frF \cap B_{r/2,L^2(\Omega)}(\rho))$, we identify $\vu_0 = \vu$ and $\rho_0 = \rho$ a.e.~in $\Omega$. 
\end{proof}

The following proposition is concerned with the open problem \labelcref{open:filtering}. We show that if a  density filter is used, then there exists a sequence of the \emph{unfiltered} discretized material distributions that converges \emph{strongly} in $L^s(\Omega)$ for any $s \in [1,\infty)$.
 
\begin{proposition}[Strong convergence of $\rho_h$ in $L^s(\Omega)$, $s\in[1,\infty)$]
\label{prop:FE:filtering-rho}
 Suppose that the conditions of \cref{th:FEexistence-filtering2} hold.  Then there exists a subsequence (up to relabelling) of  global FE minimizers  of \cref{ictoFh}, that satisfy, for any $s \in [1,\infty)$,
\begin{align}
\rho_h \to \rho \; \text{strongly in} \; L^s(\Omega). 
\end{align}
\end{proposition}

In \cref{fig:relations}, we provide the structure of the proof. The first step is to prove the two downward limits, as $\epsilon \to 0$, of minimizers of the $\epsilon$-perturbed problem \cref{ictoe} to minimizers of the  problems \cref{cto} and \cref{ictoFh}. This is achieved via weak convergence, the convergence of the supremum value of the norms, and the Radon--Riesz property \cite[Th.~A.70]{Fonseca2006}. Next, the top arrow, as $h \to 0$, in the $\epsilon$-perturbed problem is proven using the properties of the minimization problem and the Radon--Riesz property. Then, since $\rho_{\epsilon, h}$, $\rho_\epsilon$, $\rho_h$, and $\rho$ are all bounded in $L^s(\Omega)$ uniformly in $h$ and $\epsilon$, we invoke the Osgood theorem  \cite[Ch.~4, Sec.~11, Th.~2']{Zakon2017} for exchanging limits together with the Radon--Riesz property to conclude the result. 

\begin{figure}[h!]
\centering
\begin{tikzcd}[column sep=5em, row sep=3em]
\rho_{\epsilon, h} \arrow[swap]{d}{\substack{\text{Radon--Riesz} + \cref{eq:filter:6}\\ \epsilon \to 0}}   \arrow{r}{\substack{\cref{eq:filtering:rho2} \\ h \to 0}} & \rho_\epsilon  \arrow{d}{\substack{\text{Radon--Riesz} + \cref{eq:filter:6}\\ \epsilon \to 0}}  & \cref{ictoe} \\%
\rho_h  \arrow[swap]{r}{\substack{\vspace{3mm}\\ \text{Osgood} + \text{Radon--Riesz}\\ h \to 0}}   & \rho  & \cref{ictoFh}
\end{tikzcd}
\caption{A summary of the proof of \cref{prop:FE:filtering-rho}.} \label{fig:relations}
\end{figure}

\begin{proof}[Proof of  \cref{prop:FE:filtering-rho}]
Consider the global minimizer $(\vu_\epsilon, \rho_\epsilon)$ of \cref{ictoe}. If there is more than one global minimizer, then fix one. After discretizing \cref{ictoe}, then  for each $\epsilon > 0$, by  using the same arguments as in the proof of \cref{prop:FEconvergence:filtering}, there exists a sequence such that $\vu_{\epsilon, h_\epsilon} \to \vu_\epsilon$ strongly in $H^1(\Omega)^d$ and $\rho_{\epsilon, h_\epsilon} \weakstar \rho_\epsilon$ weakly-* in $L^\infty(\Omega)$ as $h_\epsilon \to 0$  where $(\vu_{\epsilon, h_\epsilon}, \rho_{\epsilon, h_\epsilon})$ is the (possibly non-unique) global minimizer of the FE discretized optimization problem \cref{ictoe} with mesh size $h_\epsilon$. 

By \cref{lem:epsilon-weak}, there exists a sequence (up to relabelling) $(\vu_\epsilon, \rho_\epsilon)$ such that $\rho_\epsilon \weak \rho$ weakly in $L^2(\Omega)$ and $\vu_\epsilon \weak \vu$ weakly in $H^1(\Omega)^d$ as $\epsilon \to 0$. We now strengthen the convergence of material distribution to strong convergence in $L^2(\Omega)$ as $\epsilon \to 0$.  Suppose that for any $\epsilon > 0$ under a certain threshold, 
\begin{align}
\| \rho \|_{L^2(\Omega)} < \| \rho_\epsilon \|_{L^2(\Omega)}.
\end{align}
This would imply that
\begin{align}
J(\vu, \rho) + \frac{\epsilon}{2}\| \rho \|^2_{L^2(\Omega)}  < J(\vu_\epsilon, \rho_\epsilon) + \frac{\epsilon}{2} \| \rho_\epsilon \|^2_{L^2(\Omega)}. 
\end{align}
However, this would be a contradiction  with the  definition of $(\vu_\epsilon, \rho_\epsilon)$. Thus,  for the sequence such that $\rho_\epsilon \weak \rho$ in $L^2(\Omega)$, 
\begin{align}
\limsup_{\epsilon \to 0} \| \rho_\epsilon \|_{L^2(\Omega)} \leq \| \rho \|_{L^2(\Omega)}.
\label{eq:filter:6}
\end{align}
Since $L^2(\Omega)$ is a uniformly convex Banach space, then by a generalized result of the Radon--Riesz property \cite[Th.~A.70]{Fonseca2006}, we have that $\rho_\epsilon \to \rho$ strongly in $L^2(\Omega)$. A similar argument gives,  for each $h > 0$,  $\rho_{ \epsilon_h , h} \to \rho_h$ strongly in $L^2(\Omega)$,  as  $\epsilon_h \to 0$. 

The next step is to show that, for each $\epsilon > 0$, there exists a subsequence such that  $\|\rho_{\epsilon, h_\epsilon}\|_{L^2(\Omega)} \to   \|\rho_\epsilon\|_{L^2(\Omega)}$ as $h_\epsilon \to 0$.  The deduction of the result is similar to the proof of \cref{prop:FE:regularization}. By \labelcref{ass:dense}, there exists a sequence of functions $(\hat{\rho}_{\epsilon, h_\epsilon})$ such that $\hat{\rho}_{\epsilon, h_\epsilon} \to \rho_\epsilon$ strongly in $L^2(\Omega)$.  Moreover, by (Case 2) in \cref{lem:primal:strong:uh} the corresponding sequence of elastic displacements $(\hat{\vu}_{\epsilon, h_\epsilon})$ such that the pair $(\hat{\vu}_{\epsilon, h_\epsilon},\hat{\rho}_{\epsilon, h_\epsilon})$ satisfies the PDE constraint in \cref{ictoe} also satisfies $\hat{\vu}_{\epsilon, h_\epsilon} \to \vu_\epsilon$ strongly in $H^1(\Omega)^d$. Thanks to the strong convergence, for sufficiently small ${h_\epsilon}$, we have that $(\hat{\vu}_{\epsilon, h_\epsilon},\hat{\rho}_{\epsilon, h_\epsilon}) \in (\vXD \cap B_{r/2,H^1(\Omega)}(\vu)) \times (\frFh \cap B_{r/2,L^2(\Omega)}(\rho))$. Hence
\begin{align}
J_\epsilon(\vu_\epsilon, \rho_\epsilon) \leq J_{\epsilon, {h_\epsilon}}(\vu_{\epsilon, h_\epsilon}, \rho_{\epsilon, h_\epsilon}) \leq J_{\epsilon, {h_\epsilon}}(\hat{\vu}_{\epsilon, h_\epsilon},\hat{\rho}_{\epsilon, h_\epsilon}),
\end{align}
where $J_{\epsilon, h}(\vv, \eta) = l_h(\vv) + \frac{\epsilon}{2} \| \eta\|^2_{L^2(\Omega)}$. In turn this implies that
\begin{align}
\begin{split}
&\frac{\epsilon}{2} \left| \|\rho_\epsilon\|^2_{L^2(\Omega)} - \| \rho_{\epsilon, h_\epsilon}\|^2_{L^2(\Omega)} \right|\\
& \indent \leq |J_{\epsilon, h}(\hat{\vu}_{\epsilon, h_\epsilon},\hat{\rho}_{\epsilon, h_\epsilon}) - J_\epsilon(\vu_\epsilon, \rho_\epsilon)| 
+ |l(\vu_\epsilon) - l_h(\vu_{\epsilon, h_\epsilon})|.
\end{split} \label{eq:filtering:strong1}
\end{align}
Label the two terms on the right-hand side of \cref{eq:filtering:strong1} by (I) and (II), respectively. Thanks to the strong convergence of $(\hat{\vu}_{\epsilon, h_\epsilon},\hat{\rho}_{\epsilon, h_\epsilon})$ in $H^1(\Omega)^d \times L^2(\Omega)$ to $(\vu_\epsilon, \rho_\epsilon)$ then $\mathrm{(I)} \to 0$. Similarly due to \labelcref{ass:fh} and \cref{cor:primal:strong:uh:filtering}, we have that $\mathrm{(II)} \to 0$. Therefore, $\|\rho_{\epsilon, h_\epsilon}\|^2_{L^2(\Omega)} \to \|\rho_\epsilon\|^2_{L^2(\Omega)}$ and consequently $\|\rho_{\epsilon, h_\epsilon}\|_{L^2(\Omega)} \to \|\rho_\epsilon\|_{L^2(\Omega)}$.

 We have shown that $\lim_{\epsilon_{h_\epsilon} \to 0}\|\rho_{\epsilon_{h_\epsilon}, h_\epsilon} \|_{L^2(\Omega)} = \|\rho_{h} \|_{L^2(\Omega)}$ for each $h$ and  $\lim_{h_\epsilon \to 0}  \|\rho_{\epsilon, h_\epsilon} \|_{L^2(\Omega)} =  \|\rho_{\epsilon} \|_{L^2(\Omega)}$ for each $\epsilon$. Thus after a diagonalization argument, we may extract a subsequence $(h_j, \epsilon_j)$, $j \in \mathbb{N}$ such that $\lim_{\epsilon_j, h_j \to 0} \| \rho_{\epsilon_j, h_j} \|_{L^2(\Omega)} = L$ for some $L \geq 0$. Dropping the subsequence subscript $j$, then an application of the Osgood Theorem \cite[Ch.~4, Sec.~11, Th.~2']{Zakon2017} for exchanging limits reveals that the limits  $ \lim_{h \to 0}  \lim_{\epsilon \to 0}\| \rho_{\epsilon,h} \|_{L^2(\Omega)}$ and $\lim_{\epsilon \to 0} \lim_{h \to 0} \| \rho_{\epsilon,h} \|_{L^2(\Omega)}$ exist and 
\begin{align}
L  =  \lim_{\epsilon \to 0} \lim_{h \to 0} \| \rho_{\epsilon,h} \|_{L^2(\Omega)} =  \lim_{h \to 0}  \lim_{\epsilon \to 0}\| \rho_{\epsilon,h} \|_{L^2(\Omega)}. \label{eq:filtering:rho1}
\end{align}
By first taking the inner limit and then the outer limit in the second term and only the inner limit in the third term of \cref{eq:filtering:rho1}, we find that
\begin{align}
\| \rho \|_{L^2(\Omega)} = \lim_{h \to 0} \| \rho_{h} \|_{L^2(\Omega)}. \label{eq:filtering:rho2}
\end{align}
Since $L^2(\Omega)$ is a uniformly convex Banach space, $\rho_{h} \weak \rho$ weakly in $L^2(\Omega)$ and \cref{eq:filtering:rho2} holds, then by the Radon--Riesz property \cite[Th.~A.70]{Fonseca2006}, we have that $\rho_{h} \to \rho$ strongly in $L^2(\Omega)$. \cref{lem:density:strong} implies that $\rho_{h} \to \rho$ strongly in $L^s(\Omega)$ for any $s \in [1,\infty)$. 
\end{proof}


\begin{corollary}
\label{cor:icto2cto:3}
 Suppose that the conditions in \cref{th:FEexistence-filtering2} hold.  Then, for sufficiently small $h$, there exists a subsequence of  global  FE  minimizers $(\vu_h,\rho_h)$ of \cref{ictoFh} that are also  local  FE  minimizers of \cref{ctoh}.
\end{corollary}
\begin{proof}
 \cref{cor:primal:strong:uh:filtering} and \cref{prop:FE:filtering-rho}  imply that there exists a subsequence $(\vu_h,\rho_h)$ such that  $(\vu_h, \rho_h) \to (\vu,\rho)$ strongly in $H^1(\Omega)^d \times L^2(\Omega)$.  By the definition of strong convergence, there exists an $\bar h$ such that for all $h \leq \bar h$, $\| \vu - \vu_h\|_{H^1(\Omega)} + \| \rho - \rho_h\|_{L^2(\Omega)} \leq r/4$. Thus the basin of attraction constraint is not active and the sequence of  global FE minimizers  of \cref{ictoFh} are also  local FE  minimizers of \cref{ctoh}.
\end{proof}

We now have the sufficient ingredients to prove \cref{th:FEexistence-filtering2}.

\begin{proof}[Proof of \cref{th:FEexistence-filtering2}]
By \cref{prop:FEexistence:filtering} and \cref{cor:weakstar}, there exists a sequence of  global FE minimizers  $(\vu_h, \rho_h)$ minimizing \cref{ictoFh} that satisfy $\vu_h \weak \vu$ weakly in $H^1(\Omega)^d$ and $\rho_h \weakstar \rho$ weakly-* in $L^\infty(\Omega)$. In \cref{cor:primal:strong:uh:filtering}, we show that there exists a subsequence (up to relabelling) such that $\vu_h \to \vu$ strongly in $H^1(\Omega)$. Then, in \cref{prop:FE:filtering-rho}, we show that there exists a subsequence (up to relabelling) such that $\rho_h \to \rho$ strongly in $L^s(\Omega)$ for any $s \in [1,\infty)$. Finally, in \cref{cor:icto2cto:3}, we show that a subsequence of the global FE minimizers are also local FE minimizers of \cref{ctoh}. 
\end{proof}

The final result we consider is the proof of \cref{th:FEexistence-filtering} concerning the strong covergence of the filtered material distributions in $W^{1,q}(\Omega)$, $q \in [1,\infty)$. The proof of this result follows a similar pattern to the proof of \cref{prop:FE:filtering-rho}. Essentially, the goal is to prove strong convergence of an $\epsilon$-perturbed problem, then by proving convergence in the value of the norms, weak convergence, and utilizing the  Radon--Riesz property, we deduce strong convergence.


\begin{proof}[Proof of \cref{th:FEexistence-filtering}]

 Recall that $(\vu, F(\rho))$ is an isolated local minimizer of \cref{fcto}.  Let $\frho \coloneqq F(\rho)$ and $\frho_h \coloneqq F_h(\rho_h)$. \cref{th:FEexistence-filtering2} implies the existence of a sequence of local FE minimizers of \cref{ctoh} such that $\vu_h \to \vu$ strongly in $H^1(\Omega)^d$ and $\rho_h \to \rho$ strongly in $L^s(\Omega)$ for any $s \in [1,\infty)$. Subsequently by \cref{lem:filteringh}, $\frho_h \to \frho$ strongly in $L^s(\Omega)$ for any $s \in [1,\infty)$. 

Let $\tilde{r}$ denote the radius of the basin of attraction. Denote the  global  minimizers of the following optimization problem by $(\vu_\epsilon, \frho_\epsilon)$: 
\begin{align}
\begin{split}
\begin{cases}
\min\limits_{\vv,\feta} \tilde{J}(\vv,\feta) + \frac{\epsilon}{q} \| \nabla \feta \|^q_{L^q(\Omega)}, \\
\text{where} \;\; (\vv,\feta) \in (\Hu \cap B_{\tilde{r}/2,H^1(\Omega)}(\vu)) \times (\ffrF \cap B_{\tilde{r}/2,W^{1,q}(\Omega)}(\frho))\\
\text{subject to} \;\; \tilde{a}(\vu_\epsilon, \vv; \frho_\epsilon) = l(\vv) \;\; \text{for all} \; \vv \in \Hu.
\end{cases}
\end{split} \tag{$\text{$\tilde{\text{I}}$-SIMP${}^F_\epsilon$}$} \label{ifctoe}
\end{align} 
The existence of a local minimizer follows by a calculus of variations argument similar to the proof of \cref{lem:eps-existence}. Consider the sequence of filtered material distribution  global  minimizers $\frho_\epsilon$ as $\epsilon \to 0$. Since $ \ffrF  \cap B_{\tilde{r}/2,W^{1,q}(\Omega)}(\frho)$ is a norm-closed, bounded and convex subset of the reflexive Banach space $W^{1,q}(\Omega)$, then there exists a subsequence (not relabeled) such that, as $\epsilon \to 0$,
\begin{align}
\frho_\epsilon \weak \frho_0 \in  \ffrF  \cap B_{\tilde{r}/2,W^{1,{ q }}(\Omega)}(\frho) \;\; \text{weakly in $W^{1,{ q }}(\Omega)$}.
\end{align}
Since  $(\vu, \frho)$ is the unique minimizer of $J$ in $(\Hu \cap B_{\tilde{r}/2,H^1(\Omega)}(\vu)) \times (\ffrF \cap B_{\tilde{r}/2,W^{1,{ q }}(\Omega)}(\frho))$,  then using the same arguments as in the proof of \cref{lem:epsilon-weak}, we may indentify $\frho_0 = \frho$ a.e.

We now wish to strengthen the convergence to strong convergence in $W^{1,{ q }}(\Omega)$.  Suppose that for any $\epsilon > 0$ under a certain threshold, 
\begin{align}
\| \nabla \frho \|_{L^{ q }(\Omega)} < \| \nabla \frho_\epsilon \|_{L^{ q }(\Omega)}.
\end{align}
This would imply that
\begin{align}
\tilde{J}(\vu, \frho) + \frac{\epsilon}{{ q }}\| \nabla\frho \|^p_{L^{ q }(\Omega)}  < \tilde{J}(\vu_\epsilon, \frho_\epsilon) + \frac{\epsilon}{{ q }} \| \nabla \frho_\epsilon \|^p_{L^{ q }(\Omega)},
\end{align}
 as $\tilde{J}(\vu_\epsilon, \frho_\epsilon) \geq \tilde{J}(\vu, \frho)$.  However, this would be a contradiction  with the  definition of $(\vu_\epsilon, \frho_\epsilon)$. Thus  for the sequence such that $\tilde{\rho}_\epsilon \weak \tilde{\rho}$ in $W^{1,{ q }}(\Omega)$ and thus $\tilde{J}_\epsilon(\vu_\epsilon, \frho_\epsilon) \to \tilde{J}(\vu, \frho)$, 
\begin{align}
\limsup_{\epsilon \to 0} \| \nabla \frho_\epsilon \|_{L^{ q }(\Omega)} \leq \| \nabla \frho \|_{L^{ q }(\Omega)}.
\label{eq:filtering3}
\end{align}
Since $L^{ q }(\Omega)$ is a uniformly convex Banach space for any ${ q } \in (1,\infty)$, then by a generalized result of the Radon--Riesz property \cite[Th.~A.70]{Fonseca2006}, we have that $\nabla \frho_\epsilon \to \nabla \frho$ strongly in $L^{ q }(\Omega)$. By an application of the Rellich--Kondrachov theorem, \labelcref{ass:filtering-box}, and \cref{lem:density:strong} we deduce that $\frho_\epsilon \to \frho$ strongly in $W^{1,{ q }}(\Omega)$.   For each $h > 0$, a similar argument implies the existence of a sequence $\epsilon_h \to 0$ that satisfies $\frho_{\epsilon_h, h} \to \frho_h$  strongly in $W^{1,{ q }}(\Omega)$. 

 With only small changes to the proof of \cref{prop:FE:regularization} we see that, for each $\epsilon > 0$, there exists a subsequence $h_\epsilon \to 0$ such that $\|\frho_{\epsilon, h_\epsilon}\|_{W^{1,q}(\Omega)} \to   \|\frho_\epsilon\|_{W^{1,q}(\Omega)}$. 

 We have shown that $\lim_{\epsilon_{h_\epsilon} \to 0}\|\frho_{\epsilon_{h_\epsilon}, h_\epsilon} \|_{W^{1,{ q }}(\Omega)} = \|\frho_{h} \|_{W^{1,{ q }}(\Omega)}$ for each $h$ and  $\lim_{h_\epsilon \to 0}  \|\rho_{\epsilon, h_\epsilon} \|_{W^{1,{ q }}(\Omega)} =  \|\rho_{\epsilon} \|_{W^{1,{ q }}(\Omega)}$ for each $\epsilon$ in their respective sequences. Thus after a diagonalization argument, we may extract a subsequence $(h_j, \epsilon_j)$, $j \in \mathbb{N}$ such that $\lim_{\epsilon_j, h_j \to 0} \| \rho_{\epsilon_j, h_j} \|_{W^{1,{ q }}(\Omega)} = L$ for some $L \geq 0$. Dropping the subsequence subscript $j$, then an application of the Osgood Theorem \cite[Ch.~4, Sec.~11, Th.~2']{Zakon2017} for exchanging limits reveals that the limits $ \lim_{h \to 0}  \lim_{\epsilon \to 0}\| \rho_{\epsilon,h} \|_{W^{1,{ q }}(\Omega)}$ and $\lim_{\epsilon \to 0} \lim_{h \to 0} \| \rho_{\epsilon,h} \|_{W^{1,{ q }}(\Omega)}$ exist and 
\begin{align}
L  =  \lim_{\epsilon \to 0} \lim_{h \to 0} \| \frho_{\epsilon,h} \|_{W^{1,{ q }}(\Omega)} =  \lim_{h \to 0}  \lim_{\epsilon \to 0}\| \frho_{\epsilon,h} \|_{W^{1,{ q }}(\Omega)}. \label{eq:filtering1}
\end{align}
Thus by first taking the inner limit and then the outer limit in the second term and only the inner limit in the third term of \cref{eq:filtering1}, we find that
\begin{align}
\| \frho \|_{W^{1,{ q }}(\Omega)} = \lim_{h \to 0} \| \frho_h \|_{W^{1,{ q }}(\Omega)}. \label{eq:filtering2}
\end{align}
Since $W^{1,{ q }}(\Omega)$ is a uniformly convex Banach space for any ${ q } \in (1,\infty)$, $\frho_h \weak \frho$ weakly in $W^{1,{ q }}(\Omega)$ and \cref{eq:filtering2} holds, then by the Radon--Riesz property \cite[Th.~A.70]{Fonseca2006}, we have that $\frho_h \to \frho$ strongly in $W^{1,{ q }}(\Omega)$. Thus $F_h(\rho_h) \to F(\rho)$ strongly in $W^{1,{ q }}(\Omega)$. 
\end{proof} 


%

\section{Conclusions and future directions}
\label{sec:conclusions}

In this work we studied the convergence of a conforming FE discretization of the SIMP model for the linear elasticity compliance topology optimization problem. To ensure existence, we considered two types of  regularization  methods: $W^{1,p}$-type  penalty methods  and density filtering. The nonconvexity of the optimization problem was handled by fixing any isolated local or global minimizer and introducing a modified optimization problem with the chosen isolated minimizer as its unique  global minimizer.  We then showed that there exists a sequence of discretized local FE minimizers that converges to the infinite-dimensional minimizer in the appropriate norms.  The  elastic  displacement and material distribution converge strongly in $H^1(\Omega)^d$ and $L^p(\Omega)$, for any $p \in [1,\infty)$, respectively, for either a $W^{1,p}$-type  penalty  or a density filter. If a $W^{1,p}$-type  penalty  is used, we further deduced that there exists a sequence such that $\rho_h \to \rho$ strongly in $W^{1,p}(\Omega)$. In contrast, if a density filter is used then the filtered material distribution converges strongly in $W^{1,q}(\Omega)$, $q \in [1,\infty)$, provided the discretization for the filtered material distribution is $W^{1,q}(\Omega)$-conforming. Thanks to the strong convergence, a subsequence of the local FE minimizers of the modified optimization problem are also minimizers of the original optimization problem.

We now outline some future directions. 
\subsubsection*{Rates of convergence}

To the best of the author's knowledge, there are currently no results that prove the rate of convergence of a sequence of local FE minimizers to an infinite-dimensional local minimizer. Numerical studies compare the convergence of local FE minimizers to a heavily refined local FE minimizer, which serves as a reference solution. Ultimately relying solely on this approach may yield misleading outcomes. It is of practical importance to know the reduction in the error when reducing the mesh size. We hope that the techniques used to prove the strong convergence results in this work will contribute significantly in this endeavour. 

\subsubsection*{Adaptive mesh refinement}

Topology optimization is a field that significantly benefits from adaptive mesh refinement. Intuitively, refining in regions where $0 < \rho \leq 1$ a.e.~reduces the error more dramatically than in areas where $\rho = 0$ a.e. These strategies enable one to achieve smaller errors with fewer degrees of freedom, thereby cutting down the overall computational cost. Notable progress has been made in this direction cf.~\cite{Wang2007, Nana2016, Stainko2006, Costa2003, Salazar2018}. However, none of these cited work prove that their proposed mesh adaptivity strategies are guaranteed to converge to an infinite-dimensional local minimizer in the limit. There is a possibility that cells in the mesh, where the error is nonzero, may not meet the refinement criteria and consequently may never be refined, hindering the convergence of the local FE minimizer in the limit. We hope that by utilizing the framework introduced in this work, one may leverage the results in \cite{Morin2008, Kreuzer2018, Kawecki2021} to devise mesh adaptive algorithms that provably converge in the limit.

\subsubsection*{Locking}
\label{sec:locking}
A common criticism of the primal formulation of classical linear elasticity (without the topology optimization component) is that as the material becomes incompressible $(\lambda \to \infty)$, the operator norm becomes unbounded. This can result in \emph{locking}, the phenomenon where the approximation has sub-optimal convergence, or even appear to diverge, for large ranges of the mesh size $h$ (before the asymptotic regime kicks in). In the primal formulation for linear elasticity, it is often possible to prove an error estimate of the form, for $\vu \in H^s(\Omega)^d$ \cite[Eq.~(2.9a)]{Ainsworth2022}:
\begin{align}
\| \vu - \vu_h \|_{H^1(\Omega)} \leq C(\lambda) h^{\min(k, s-1)} \| \vu \|_{H^s(\Omega)},
\end{align} 
where $k$ denotes the polynomial degree of the FE discretization of $\vu$. Moreover $C(\lambda) \to \infty$ as $\lambda \to \infty$, which corresponds to the incompressible limit. Asymptotically, the error should reduce at rate of $h^{\min(k, s-1)}$. However, if $C(\lambda)$ is extremely large, the constant will dominate the error resulting in locking.  Locking can be alleviated by using a high-order element for $\vu$, e.g.~quartics or higher \cite{Ainsworth2022} or by using alternatives to the primal formulation such as the pressure and symmetric stress formulations \cite[Ch.~VI.3]{Braess2007}.

The appearance of locking is particularly troublesome when considering the SIMP model. The primal formulation with low-order elements is used in almost all implementations of the SIMP model, e.g.~\cite{Andreassen2011}. Moreover, since there are almost no known formulae for local minimizers of SIMP models, there are few comprehensive studies on FE convergence. We conjecture that locking occurs near the incompressible limit in SIMP models and, consequently, the computed discretized minimizers have a high error even for small mesh sizes.

Rigorously describing locking is challenging and the results of \cite{Ainsworth2022} and \cite[Ch.~VI.3]{Braess2007} do not (immediately) extend to the SIMP model. Thus any proof of non-locking formulations was beyond the scope of this work. Nevertheless, the new convergence results presented in this study aim to lay some groundwork for addressing the locking issue in future research. A comprehensive mathematical understanding of locking in the SIMP model would have significant implications for future implementations.

\appendix

\section{Supplementary proofs}
\label{sec:app:proofs}

\begin{proof}[Proof of \cref{prop:filter:ass}]

Since  $\rho \in \frF$ and thus $f, \rho \geq 0$ a.e.~then $F(\rho)(x) \geq 0$ for all $x \in \Omega$. Moreover,
\begin{align}
|F(\rho)(x)| = \left| \int_\Omega f(x-y) \rho(y) \, \dy \right| \leq \| f \|_{L^1(\mathbb{R}^d)} \| \rho \|_{L^\infty(\Omega)} \leq 1. \label{ass:filtering6}
\end{align}
Thus \labelcref{ass:filtering-box} is satisfied. Moreover, \cref{ass:filtering6} implies that
\begin{align}
\| F(\rho) \|_{L^\infty(\Omega)} = \left \| \int_\Omega f(x-y) \rho(y) \, \dy \right \|_{L^\infty(\Omega)} \leq \|  f \|_{L^\infty(\mathbb{R}^d)} \leq C < \infty. 
\end{align} 
It can be shown that $h(x,y) \coloneqq f(x-y)\rho(y)$ satisfies the condition for the Leibniz integral rule and, therefore,
\begin{align}
\| \nabla F(\rho) \|_{L^\infty(\Omega)} =  \left \| \int_\Omega \nabla_x[f(x-y)] \rho(y) \, \dy \right \|_{L^\infty(\Omega)} \leq \| \nabla f  \|_{L^\infty(\mathbb{R}^d)} \leq C < \infty. 
\end{align}
This implies that $F(\rho) \in W^{1,\infty}(\Omega)$. Hence, \labelcref{ass:filtering2} is satisfied. 

Consider a weakly-* converging sequence $\rho_n \weakstar \rho \in  \frF$   weakly-* in $L^\infty(\Omega)$. This defines a bounded sequence of Lipschitz functions $F(\rho_n)$. Thus, by the Arzel\`a--Ascoli theorem \cite[Th.~1.33]{Adams2003}, there exists a limit $F_0 \in L^\infty(\Omega)$ such that a subsequence (up to relabelling) satisfies $F(\rho_n) \to F_0$ uniformly in $\Omega$. Hence, $F(\rho_n) \to F_0$ strongly in $L^\infty(\Omega)$. Moreover,
\begin{align}
F(\rho_n)(x) = \int_\Omega f(x-y) \rho_n(y) \, \dy \to  \int_\Omega f(x-y) \rho(y) \, \dy = F(\rho)(x). 
\end{align}
Hence $F(\rho_n) \to F(\rho)$ pointwise. Thus, by the uniqueness of limits, $F_0 = F(\rho)$ and $F(\rho_n) \to F(\rho)$ strongly in $L^\infty(\Omega)$. A similar argument reveals that if $\rho_n \weak \rho$ weakly in $L^s(\Omega)$, $s \in [1,\infty)$, then $F(\rho_n) \to F(\rho)$ strongly in $L^s(\Omega)$. Thus \labelcref{ass:filtering} is satisfied.

Note that $F(\rho) + t(F(\rho) - F(\eta)) = F(\rho+t(\eta-\rho)) \in \ffrF$ since $\rho+t(\eta-\rho) \in \frF$. Hence $\ffrF$ is a convex space. Consider any sequence such that $F(\rho_n) \to F_0$ strongly in $W^{1,s}(\Omega)$ for any $s \in [1,\infty]$. This defines a bounded sequence $(\rho_n)$  in $L^\infty(\Omega)$.  Hence, since $\frF$ is a norm-closed, bounded, and convex subset of $L^\infty(\Omega)$, then by the Banach--Alaoglu theorem, there exists an $\rho_0 \in \frF$ such that a subsequence (up to relabelling) satisfies $\rho_n \weakstar \rho_0$ weakly-* in $L^\infty(\Omega)$. As previously shown, this implies that $F(\rho_n) \to F(\rho_0)$ strongly in $L^\infty(\Omega)$. Hence, $F_0 = F(\rho_0)$ and, therefore, $F_0 \in \ffrF$. Hence, $\ffrF$ is a norm-closed, bounded, and convex subset of $W^{1,\infty}(\Omega)$. A similar argument reveals that $\ffrF$ is a norm-closed, bounded, and convex subset of $W^{1,s}(\Omega)$ for any $s \in [1,\infty]$. Therefore, \labelcref{ass:filtering4} is satisfied.
\end{proof}

\begin{proof}[Proof of \cref{lem:density:strong}]
By H\"older's inequality $\| \eta - \eta_{ j } \|_{L^1(\Omega)} \leq |\Omega|^{1/p} \| \eta - \eta_{ j } \|_{L^p(\Omega)}$. Thus $\eta_{ j } \to \eta$ strongly in $L^1(\Omega)$. Therefore, for any $s \in (1,\infty)$,
\begin{align}
\int_{\Omega} | \eta - \eta_{ j }|^s \dx = \int_{\Omega}  | \eta - \eta_{ j }|^{s-1}  | \eta - \eta_{ j }|\dx \leq 1^{s-1} \| \eta - \eta_{ j }\|_{L^1(\Omega)},
\end{align}
which implies that $\eta_{ j } \to \eta$ strongly in $L^s(\Omega)$ for any $s \in [1,\infty)$.
\end{proof}

\begin{proof}[Proof of \cref{lem:filteringh}]
 Fix an $\eta \in L^q(\Omega)$  such that $\eta_h \weak \eta \in \frF$ weakly in $L^q(\Omega)$ ($\eta_h \weakstar \eta$ weakly-* in $L^\infty(\Omega)$ if $q=\infty$) for $q \in [1,\infty]$. Then,  
\begin{align}
\| F(\eta) - F_h(\eta_h) \|_{L^q(\Omega)} \leq \| F(\eta) - F_h(\eta) \|_{L^q(\Omega)} + \|F_h(\eta) - F_h(\eta_h) \|_{L^q(\Omega)}.  \label{eq:app:f1}
\end{align}
The first term on the right-hand side of \cref{eq:app:f1} tends to zero thanks to assumption \labelcref{ass:filtering-projection}. For the second term we note that, by assumptions \labelcref{ass:filtering-projection2} and \labelcref{ass:filtering3},
\begin{align}
 \|F_h(\eta) - F_h(\eta_h) \|_{L^q(\Omega)} \leq C \| F(\eta) - F(\eta_h) \|_{L^q(\Omega)}
 \label{eq:app:f2}
\end{align}
for a constant $C<\infty$ independent of $h$. The right-hand side of \cref{eq:app:f2} tends to zero thanks to \labelcref{ass:filtering}. Hence $\eta_h \weak \eta$ weakly in $L^q(\Omega)$ ($\eta_h \weakstar \eta$ weakly-* in $L^\infty(\Omega)$ if $q=\infty$) implies that $F_h(\eta_h) \to F(\eta)$ strongly in $L^q(\Omega)$. 
\end{proof}

\begin{proof}[Proof of \cref{lem:weakpde}]
 We prove the result for (Case 1) and note that (Case 2) follows with only a couple of small modifications.  For any $\vv \in \Hu$ and $\vv_h \in \vXD$, consider
\begin{align}
\begin{split}
a_h(\hat{\vu}_h, \vv; \hat{\rho}_h) &= a_h(\hat{\vu}_h,  \vv_h; \hat{\rho}_h) + a_h(\hat{\vu}_h, \vv - \vv_h; \hat{\rho}_h)\\
&= l_h(\vv_h) + a_h(\hat{\vu}_h, \vv -\vv_h; \hat{\rho}_h).
\end{split} \label{eq:app2} 
\end{align}
By \labelcref{ass:dense} we may choose $\vv_h$ such that $\vv_h \to \vv$ strongly in $H^1(\Omega)^d$. Thus
\begin{align}
 l_h(\vv_h) \to l(\vv).
\end{align}
due to the linearity of $l(\cdot)$, \labelcref{ass:fh}, and \labelcref{ass:dense}. Moreover, 
\begin{align}
|a_h(\hat{\vu}_h, \vv - \vv_h; \hat{\rho}_h)| \leq \| k(\hat{\rho}_h)\|_{L^\infty(\Omega)} \| \tE \hat{\vu}_h\|_{L^2(\Omega)} \|  \tE(\vv - \vv_h) \|_{L^2(\Omega)} \to 0,  \label{eq:app3} 
\end{align}
thanks to the boundedness of $k(\hat{\rho}_h)$, $\tE\hat{\vu}_h$ and \labelcref{ass:dense}.  The boundedness of $\tE\hat{\vu}_h$  follows from the a priori estimates in the PDE of \cref{ctoh}.  Hence,
\begin{align}
a_h(\hat{\vu}_h, \vv; \hat{\rho}_h) \to l(\vv).  \label{eq:app4}
\end{align}
We note that
\begin{align}
\begin{split}
&|a_h(\hat{\vu}_h, \vv; \hat{\rho}_h) - a(\hat{\vu}, \vv; \hat{\rho})| \\
& \indent \leq | a(  \hat{\vu}_h - \hat{\vu}, \vv; \hat{\rho})| + |a_h(\hat{\vu}_h, \vv; \hat{\rho}_h) - a(\hat{\vu}_h, \vv; \hat{\rho})|. 
\end{split}\label{eq:app6}
\end{align}
The first term on the right-hand side tends to zero thanks to weak convergence of $\hat{\vu}_h$ to $\hat{\vu}$. For the second term we see that
\begin{align}
|a_h(\hat{\vu}_h, \vv; \hat{\rho}_h) - a(\hat{\vu}_h, \vv; \hat{\rho})| \leq \|  (k(\hat{\rho}_h) - k(\hat{\rho}))\tE \vv \|_{L^2(\Omega)} \| \tE\vu_h\|_{L^2(\Omega)}. \label{eq:app7}
\end{align}
By assumption  $\hat{\rho}_h \weak \hat{\rho}$ in $W^{1,q}(\Omega)$ for some $q \in [2,\infty)$ and, by the Rellich--Kondrachov theorem and \cref{lem:density:strong}, $\hat{\rho}_h \to \hat{\rho}$ strongly in $L^2(\Omega)$.  Thus by utilizing \cref{lem:density:strong}, for any smooth test function $\vv \in C_c^\infty(\Omega)^d$,
\begin{align}
a_h(\hat{\vu}_h, \vv; \hat{\rho}_h) \to  a(\hat{\vu}, \vv; \hat{\rho}). \label{eq:app8}
\end{align}
\cref{eq:app8} follows for any $\vv \in \Hu$ by the density of smooth and compactly supported functions in $\Hu$. Hence, \cref{eq:app4} and \cref{eq:app8} imply that, for any $\vv \in \Hu$,
\begin{align}
a(\hat{\vu}, \vv; \hat{\rho}) = l(\vv). 
\end{align}
\end{proof}

\begin{proof}[Proof of \cref{lem:primal:strong:uh}]
 We prove the result for (Case 1) and note that (Case 2) follows with only a couple of small modifications.  The first step is to show that $\hat{\vu}_h \weak \vu$ in $H^1(\Omega)^d$. The sequence generated by $\hat{\rho}_h$ defines a bounded sequence of $\hat{\vu}_h$ as $h \to 0$. Since $\Hu$ is a reflexive space, then there exists a limit $\hat{\vu} \in \Hu$ and a subsequence (up to relabelling) such that $\hat{\vu}_h \weak \hat{\vu} \in \Hu$. We must now identify $\hat{\vu}$ with $\vu$.

The requirements of \cref{lem:weakpde} are satisfied and thus, for any $\vv \in \Hu$,
$a(\hat{\vu}, \vv; \rho) = l(\vv)$.  Due to the uniqueness of limits and the uniqueness of solutions of the linear elasticity problem with a fixed $\rho$ (\cref{prop:uniqueness}), we deduce that $\hat{\vu} = \vu$ a.e.~and $\hat{\vu}_h \weak \vu$ in $H^1(\Omega)^d$.

We now strengthen the convergence to strong convergence. By Korn's inequality \cite[Th.~11]{Bauer2016}, and the assumption that $\epsilon_{\text{SIMP}}>0$, there exists a $c>0$ such that
\begin{align}
\begin{split}
c \| \vu - \hat{\vu}_h \|^{2}_{H^1(\Omega)} 
&\leq a_h(\vu - \hat{\vu}_h, \vu - \hat{\vu}_h; \hat{\rho}_h)\\
&=a_h(\vu - \hat{\vu}_h, \vu; \hat{\rho}_h) + a_h(\hat{\vu}_h, \hat{\vu}_h; \hat{\rho}_h) - a_h(\vu, \hat{\vu}_h; \hat{\rho}_h)\\
&=  a_h(\vu - \hat{\vu}_h, \vu; \hat{\rho}_h) +l_{ h }( \hat{\vu}_h ) - a_h(\vu, \hat{\vu}_h; \hat{\rho}_h).
\end{split} \label{eq:app1}
\end{align}
The first term on the rightmost-hand side of \cref{eq:app1}  is bounded above as follows 
\begin{align}
\begin{split}
& a_h(\vu - \hat{\vu}_h, \vu; \hat{\rho}_h) \\
&\indent \leq a(\vu - \hat{\vu}_h, \vu;\rho)
 + \| k(\rho) - k(\rho_h) \|_{L^2(\Omega)}  \| \tE(\vu - \hat{\vu}_h) \|_{L^2(\Omega)} \|\tE \vu \|_{L^2(\Omega)}. 
\end{split}
\end{align}
The first term tends to zero  thanks to $\hat{\vu}_h \weak \vu$ weakly in $H^1(\Omega)^d$.  The second term converges to zero since  $\rho_h \to \rho$ strongly in $L^q(\Omega)$ for some $q \in [2,\infty]$ (and thus for any $q \in [1,\infty)$ by \cref{lem:density:strong}). 

The second term on the rightmost-hand side of \cref{eq:app1} tends to  $l(\vu)$  thanks to \labelcref{ass:fh} and the weak convergence of $\hat{\vu}_h$  in $H^1(\Omega)^d$.  By the same arguments as \cref{eq:app2}--\cref{eq:app3}, we note that  $a_h(\vu, \hat{\vu}_h; \hat{\rho}_h) \to l(\vu)$.  Hence, we conclude that $\hat{\vu}_h \to \vu$ strongly in $H^1(\Omega)^d$. 
\end{proof}

\printbibliography

\end{document}